\documentclass[final,1p,times]{elsarticle}
\usepackage{amssymb}
\pdfoutput=1

\usepackage{mathrsfs}
\usepackage{epsfig, graphicx}
\usepackage{latexsym,amsfonts,amsbsy,amssymb}
\usepackage{amsmath,amsthm}
\usepackage{color}
\usepackage[colorlinks, citecolor=blue]{hyperref}
\usepackage{enumerate}
\usepackage{float}
 \usepackage{bbding}

\newtheorem{theorem}{Theorem}[section]
\newtheorem{lemma}{Lemma}[section]

\newtheorem{remark}{Remark}[section]

\newtheorem{definition}{Definition}[section]

\journal{}

\begin{document}

\begin{frontmatter}



\title{Local Fourier analysis for mixed finite-element methods for the Stokes equations}


\author[A]{Yunhui He\corref{cor1}}
\ead{yunhui.he@mun.ca}

\author[A]{Scott P. MacLachlan}
\ead{smaclachlan@mun.ca}
\cortext[cor1]{Corresponding author: Yunhui He.  E-mail address: yunhui.he@mun.ca.}

\address[A]{Department of Mathematics and Statistics, Memorial University of Newfoundland, St. John's, NL A1C 5S7, Canada}

\begin{abstract}
 In this paper, we develop a local Fourier analysis of multigrid methods based on block-structured relaxation schemes for stable and stabilized mixed finite-element discretizations of  the Stokes equations, to analyze their convergence behavior. Three  relaxation schemes are considered: distributive, Braess-Sarazin, and Uzawa relaxation.  From this analysis, parameters that minimize the  local Fourier analysis smoothing factor are proposed for the stabilized methods with distributive and Braess-Sarazin relaxation. Considering the failure of the local Fourier analysis smoothing factor in predicting the true two-grid convergence factor for the stable discretization, we numerically optimize the two-grid convergence predicted by local Fourier analysis in this case. We also compare the efficiency of the presented algorithms with variants using inexact solvers.  Finally, some numerical experiments are presented to validate the two-grid and multigrid convergence factors.

\end{abstract}


\end{frontmatter}



\section{Introduction}
\label{sec:intro}
In recent years, substantial research has been devoted to efficient numerical solution of the Stokes and Navier-Stokes equations, due both to their utility as models of (viscous) fluids and their commonalities with many other physical problems that lead to saddle-point systems (see, for example \cite{elman2006finite}, and many of the other references cited here). In the linear (or linearized) case, solution of the resulting matrix equations is seen to be difficult, due to indefiniteness and the usual ill-conditioning of discretized PDEs. In the literature, block preconditioners (cf. \cite{elman2006finite} and the references therein) are widely used, due to their easy construction  from standard multigrid algorithms for scalar elliptic PDEs, such as algebraic multigrid \cite{ruge1987algebraic}. However, monolithic multigrid approaches \cite{adler2016monolithic,MR3439774,MR558216,MR2833487,MR2001083} have been shown to outperform these preconditioners when algorithmic parameters are properly chosen \cite{MR3639325,farrell2018augmented}.  The focus of this work is on the analysis of such monolithic multigrid methods in the case of stable and stabilized finite-element discretizations of the Stokes equations.

Local Fourier analysis (LFA) \cite{MR1807961,wienands2004practical} has been widely used to predict the convergence behavior of
multigrid methods, to help design relaxation schemes and choose algorithmic parameters. In general, the LFA smoothing factor provides a sharp prediction of actual multigrid convergence, see \cite{MR1807961}, under the assumption of an ``ideal'' coarse-grid correction scheme (CGC) that annihilates  low-frequency error components and leaves  high-frequency components unchanged.  In practice, the LFA smoothing and two-grid convergence  factors often  exactly match the true convergence factor of multigrid applied to a problem with periodic boundary conditions \cite{brandt1994rigorous,stevenson1990validity,MR1807961}. Recently, the validity of LFA  has been further analysed \cite{rodrigo2017validity}, extending this exact prediction to a wider class of problems. However, the LFA smoothing factor is also known to lose its predictivity of the true convergence in some cases \cite{HM2018LFALaplace,MR2840198,friedhoff2013local}.  In particular, the smoothing factor of LFA overestimates the two-grid  convergence factor for the Taylor-Hood ($Q_2-Q_1$) discretization of the Stokes equations with Vanka relaxation \cite{MR2840198}. Even for the  scalar Laplace operator, the LFA smoothing factor fails to predict the observed multigrid convergence factor for higher-order finite-element methods \cite{HM2018LFALaplace}.

Two main questions interest us here. First, we look to extend the study of \cite{MR2840198} to consider LFA of block-structured relaxation schemes for finite-element discretizations of the Stokes equations. Secondly, we consider if the LFA smoothing factor can predict the convergence factors for these relaxation schemes.   Recently, LFA for multigrid based on block-structured relaxation schemes applied to the marker-and-cell (MAC) finite-difference discretization of the Stokes equations was shown to give a good prediction of convergence \cite{NLA2147}, in contrast to the results of \cite{MR2840198}. Thus, a natural question to investigate is whether the contrasting results between  \cite{NLA2147} and \cite{MR2840198} is due to the differences in discretization or those in the relaxation schemes considered. Here, we apply the relaxation schemes of \cite{NLA2147} to the $Q_2-Q_1$ discretization from \cite{MR2840198}, as well as an ``intermediate'' discretization using stabilized $Q_1-Q_1$ approaches.

In recent decades, many block relaxation schemes have been studied and applied to many problems, including Braess-Sarazin-type relaxation schemes \cite{adler2016monolithic,braess1997efficient,MR1685458,MR3439774,MR1810326}, Vanka-type relaxation schemes \cite{adler2016monolithic,MR848451,MR3439774,MR2840198,MR2263039,MR3488076,MR2001083}, Uzawa-type relaxation schemes \cite{john2016analysis,MR3217219,MR3580776,MR2833487,MR833993}, distributive relaxation schemes \cite{bacuta2011new,MR558216,wittum1989multi,MR3071182,oosterlee2006multigrid} and other types of methods \cite{MR3427891,MR3432850}. Even though LFA has been applied to distributive relaxation \cite{MR1049395,wienands2004practical}, Vanka relaxation \cite{MR2840198,MR3488076,sivaloganathan1991use,MR1131567}, and Uzawa-type schemes \cite{MR3217219} for the Stokes equations, most of the existing LFA has been for relaxation schemes using (symmetric) Gauss-Seidel (GS) approaches, and for simple finite-difference and finite-element discretizations.  Considering modern multicore and accelerated parallel architectures, we focus on schemes based on weighted Jacobi relaxation with distributive, Braess-Sarazin, and Uzawa relaxation for common finite-element discretizations of the Stokes equations.

Some key conclusions of this analysis are as follows.  First, while the LFA smoothing factor gives a good prediction of the true convergence factor for the stabilized discretizations with distributive weighted Jacobi and Braess-Sarazin relaxation, it does not for the Uzawa relaxation (in contrast to what is seen for the MAC discretization \cite{NLA2147,MR1049395}).  For no cases does the LFA smoothing factor offer a good prediction of the true convergence behaviour for the (stable) $Q_2-Q_1$ discretization, suggesting that the discretization is responsible for the lack of predictivity, consistent with the results in \cite{HM2018LFALaplace,MR2840198}.  For both stable and stabilized discretizations, we see that standard distributive weighted Jacobi relaxation loses some of its high efficiency, in contrast to what is seen for the MAC scheme \cite{NLA2147,MR1049395} but that robustness can be restored with an additional relaxation sweep.  Exact Braess-Sarazin relaxation is also highly effective, with LFA-predicted $W(1,1)$ convergence factors of $\frac{1}{9}$ in the stabilized cases and $\frac{1}{4}$ in the stable case.  To realize these rates with inexact cycles, however, requires nested W-cycles to solve the approximate Schur complement equation accurately enough in the stabilized case, although simple weighted Jacobi on the approximate Schur complement is observed to be sufficient in the stable case.  For Uzawa-type relaxation, we see a notable gap between predicted convergence with exact inversion of the resulting Schur complement, versus inexact inversion, although some improvement is seen when replacing the approximate Schur complement with a mass matrix approximation, as is commonly used in block-diagonal preconditioners  \cite{wathen2009chebyshev,wathen1993fast,silvester1994fast}.  Overall, however, we see that distributive weighted Jacobi (DWJ) (with 2 sweeps of  Jacobi relaxation on the pressure equation)  outperforms both Braess-Sarazin relaxation (BSR) and Uzawa relaxation, for the stabilized discretizations, while DWJ and inexact BSR offer comparable performance for the stable discretization.

We organize this paper as follows. In Section \ref{sec:Q2Q1-FEM}, we introduce two stabilized $Q_1-Q_1$ and the stable $Q_2-Q_1$ mixed finite-element discretizations of the Stokes equations in two dimensions (2D). In Section \ref{sec:defi-FEM}, we first review the LFA approach, then discuss the Fourier representation for these discretizations.  In Section \ref{LFA-relaxation-FEM}, LFA is developed for DWJ, BSR, and Uzawa-type relaxation,  and optimal LFA smoothing factors are derived for the two stabilized $Q_1-Q_1$ methods with DWJ and BSR. Multigrid performance is presented to validate the theoretical results.  Section \ref{sec:Q2Q1performance} exhibits  optimized  LFA two-grid convergence factors and measured multigrid convergence
  factors for the $Q_2-Q_1$ discretization. Furthermore, a comparison of the cost and effectiveness of the  relaxation schemes is given.  Conclusions are presented in Section \ref{sec:concl-FEM}.

\section{Discretizations}
\label{sec:Q2Q1-FEM}
In this paper, we consider the Stokes equations,
\begin{eqnarray}
  -\Delta\vec{u}+\nabla p&=&\vec{f},\label{Stokes1-FEM}\\
  \nabla\cdot \vec{u}&=&0,\nonumber
\end{eqnarray}
where $\vec{u}$ is the velocity vector, $p$ is the (scalar) pressure of a viscous fluid, and $\vec{f}$ represents a (known) forcing term, together with suitable boundary conditions. Because of the nature of LFA, we validate our predictions against the problem with periodic boundary conditions on both $\vec{u}$ and $p$. Discretizations of (\ref{Stokes1-FEM}) typically lead to a linear system of the following form:
\begin{equation}\label{saddle-structure-FEM}
     Kx=\begin{pmatrix}
      A & B^{T}\\
     B & - \beta C\\
    \end{pmatrix}
        \begin{pmatrix} \mathcal{U }\\ {\rm p}\end{pmatrix}
  =\begin{pmatrix} {\rm f} \\ 0 \end{pmatrix}=b,
 \end{equation}
where $A$ corresponds to the discretized vector Laplacian, and $B$ is the negative of the discrete divergence operator. If the discretization is naturally unstable, then $C\neq 0$ is the stabilization matrix, otherwise $C=0$. In this paper, we discuss two stabilized $Q_1-Q_1$ and  the stable $Q_2-Q_1$  finite-element discretizations.

The natural  finite-element approximation of Problem (\ref{Stokes1-FEM}) is:
Find $\vec{u}_h\in \mathcal{X}^h$ and $p_h\in \mathcal{H}^{h}$  such that
\begin{equation}\label{mixed-FEM-form}
a(\vec{u}_h,\vec{v}_h) + b(p_h,\vec{v}_h)+b(q_h,\vec{u}_h) =g(\vec{v}_h),\,\, {\rm for\,\, all}\, \vec{v}_h\in\mathcal{X}_{0}^h \,\,{\rm and}\,\, q_h\in\mathcal{H}^{h},
\end{equation}
where
\begin{eqnarray*}
a(\vec{u}_h,\vec{v}_h)& = &\int_{\Omega}\nabla \vec{u}_h:\nabla \vec{v}_h,\,\,\, b(p_h,\vec{v}_h)=-\int_{\Omega} p_h \nabla \cdot \vec{v}_h,\nonumber\\
 g(\vec{v}_h) &=&\int_{\Omega} \vec{f}_h\cdot \vec{v}_h,
\end{eqnarray*}
and $\mathcal{X}^h\subset H^{1}(\Omega)$, $\mathcal{H}^{h}\subset L_2(\Omega)$ are finite-element spaces. Here, $\mathcal{X}_{0}^h \subset\mathcal{X}^h$ satisfies homogeneous Dirichlet boundary conditions in place of any non-homogenous essential boundary conditions on $\mathcal{X}^h$. Problem (\ref{mixed-FEM-form}) has a unique solution only when $\mathcal{X}^h$ and $\mathcal{H}^{h}$ satisfy an inf-sup condition (see \cite{elman2006finite,MR2373954,brezzi1988stabilized,dohrmann2004stabilized}).

\subsection{Stabilized $Q_1-Q_1$ discretizations}
The standard equal-order approximation of (\ref{mixed-FEM-form})  is  well-known to be unstable \cite{brezzi1988stabilized,elman2006finite}.  To circumvent this,  a scaled  pressure Laplacian term  can be added to (\ref{mixed-FEM-form});  for a uniform mesh with square elements of size $h$, we subtract
\begin{equation*}
  c(p_h,q_h) = \beta h^2(\nabla p_h,\nabla q_h),
\end{equation*}
for $\beta>0$.  With this, the resulting linear system is given by
\begin{equation*}
     \begin{pmatrix}
      A & B^{T}\\
     B & - \beta h^2 A_p \\
    \end{pmatrix}
        \begin{pmatrix} \mathcal{U} \\ {\rm p }\end{pmatrix}
  =\begin{pmatrix} {\rm f} \\ 0 \end{pmatrix}=b,
 \end{equation*}
where $A_p$ is the $Q_1$ Laplacian operator  for the pressure. Denote $S= BA^{-1}B^{T}$, and  $S_{\beta}= BA^{-1}B^{T} +\beta C $, where $C=h^2 A_p$. From \cite{elman2006finite}, the red-black unstable mode $\textbf{p} = \pm \bf{1}$,
can be  moved from a zero eigenvalue to a unit eigenvalue ( giving stability
without loss of accuracy) by choosing $\beta$ so that
\begin{equation}\label{beta-choosing}
 \frac{\textbf{p}^{T}S_{\beta}\textbf{p}}{\textbf{p}^{T}Q\textbf{p}}= \beta\frac{\textbf{p}^{T}C\textbf{p}}{\textbf{p}^{T}Q\textbf{p}}=1,
\end{equation}
where $Q$ is the mass matrix. Substituting the bilinear stiffness and mass matrices  into  (\ref{beta-choosing}), we find $\beta=\frac{1}{24}$.  We refer to this  method as the Poisson-stabilized discretization (PoSD).

An $L_2$ projection to stabilize  the $Q_1-Q_1$  discretization, proposed in
 \cite{dohrmann2004stabilized}, stabilizes with
\begin{equation}\label{projection-stabilized}
  C(p_h,q_h) = (p_h-\Pi_{0}p_h,q_h-\Pi_{0}q_h),
\end{equation}
where $\Pi_{0}$ is the $L_{2}$ projection from $\mathcal{H}^{h}$ into the space of piecewise constant functions on the mesh. We refer to this method as the projection stabilized discretization (PrSD).
 The $4\times 4$ element matrix $C_4$ of (\ref{projection-stabilized}) is given by
\begin{equation*}\label{prejection-4x4-element}
  C_4 = Q_4-\textbf{q}\textbf{q}^{T}h^2,
\end{equation*}
where $Q_4$ is the $4\times 4$ element mass matrix for the bilinear discretization and
$\textbf{q}=\begin{bmatrix}
  \frac{1}{4} & \frac{1}{4} & \frac{1}{4} &\frac{1}{4}
   \end{bmatrix}^{T}$.
In the projection stabilized method, we can write  $C=Q-h^2P$,
where $P$ is given by the 9-point stencil
\begin{equation*}
  P = \frac{1}{4}\begin{bmatrix}
   \frac{1}{4} &  \frac{1}{2}     & \frac{1}{4} \\
   \frac{1}{2}  &  1   & \frac{1}{2} \\
   \frac{1}{4}  &  \frac{1}{2}     & \frac{1}{4}
   \end{bmatrix}.
\end{equation*}
Applying (\ref{beta-choosing}) to $C=Q-h^2P$, we find that $\beta=1$ is the optimal choice.

\subsection{Stable $Q_2-Q_1$ discretizations}
In order to guarantee the well-posedness of the discrete system (\ref{saddle-structure-FEM}) with $C=0$, the discretization of the velocity and pressure unknowns should satisfy an inf-sup condition,
\begin{equation*}
  \inf_{q_h\neq 0} \sup_{\vec{v}_h\neq \vec{0}}\frac{|b(q_h,\vec{v}_h)|}{\|\vec{v}_h\|_{1}\|q_h\|_{0}}\geq \Gamma>0,
\end{equation*}
where $\Gamma$ is a constant. Taylor-Hood ($Q_2-Q_1$) elements are well known to be
stable \cite{MR2373954, elman2006finite}, where the basis functions associated with these elements are biquadratic
for each component of the velocity field and bilinear for the pressure.

\section{LFA preliminaries}\label{sec:defi-FEM}

\subsection{Definitions and notations}
In many cases, the LFA smoothing factor offers a good prediction of multigrid performance. Thus, we will explore the LFA smoothing factor and true (measured) multigrid convergence  for the three types of relaxations considered here. We first introduce some terminology of LFA, following \cite{MR1807961,wienands2004practical}.
We consider the following two-dimensional infinite uniform grids,
\begin{equation*}
  \mathbf{G}^{j}_{h}=\big\{\boldsymbol{x}^{j}:=(x^j_1,x^j_2)=(k_{1},k_{2})h+\delta^{j},(k_1,k_2)\in \mathbb{Z}^2\big\},
\end{equation*}
with
\begin{equation*}
\delta^{j}=\left\{
  \begin{aligned}
    &(0,0) &\text{if}\quad j=1,\\
    &(0,h/2)  &\text{if} \quad j=2, \\
    &(h/2,0)  &\text{if} \quad j=3, \\
    &(h/2,h/2) &\text{if}\quad j=4.\\
  \end{aligned}
               \right.
\end{equation*}
The coarse grids, $\mathbf{G}^j_{2h}$, are defined similarly.

\begin{figure}[H]
\centering
\includegraphics[width=6.5cm,height=5.5cm]{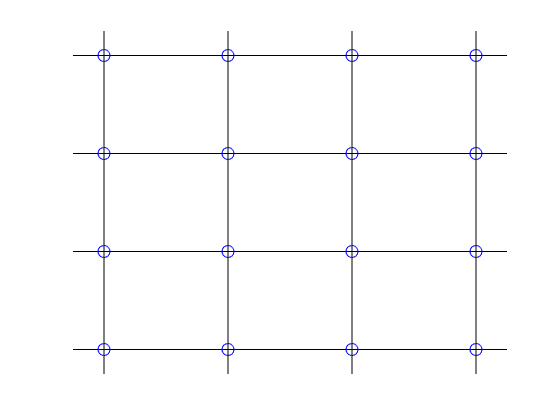}
\includegraphics[width=6.5cm,height=5.5cm]{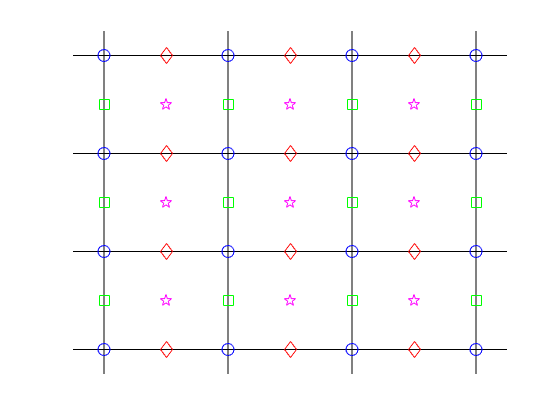}
\caption{At left, the mesh used for $Q_1$ discretization. At right, the mesh used for $Q_2$ discretization. Points marked by $\circ$ correspond to $\mathbf{G}^{1}_{h}$, those marked by $\Diamond$ correspond to $\mathbf{G}^{2}_{h}$, those marked by $\Box$ correspond to $\mathbf{G}^{3}_{h}$ and those marked by \FiveStarOpen   \, correspond to $\mathbf{G}^{4}_{h}$.} \label{Q2-Q1-mesh}
\end{figure}

Let $L_h$ be a scalar Toeplitz operator defined by its stencil acting on $\mathbf{G}^{j}_{h}$ as follows:
\begin{eqnarray}\label{defi-symbol-FEM}
  L_{h}  &\overset{\wedge}{=}& [s_{\boldsymbol{\kappa}}]_{h} \,\,(\boldsymbol{\kappa}=(\kappa_{1},\kappa_{2})\in \boldsymbol{V}); \,
  L_{h}w_{h}(\boldsymbol{x}^j)=\sum_{\boldsymbol{\kappa}\in\boldsymbol{V}}s_{\boldsymbol{\kappa}}w_{h}(\boldsymbol{x}^j+\boldsymbol{\kappa}h),
  \end{eqnarray}
with constant coefficients $s_{\boldsymbol{\kappa}}\in \mathbb{R} \,(\textrm{or} \,\,\mathbb{C})$, where $w_{h}(\boldsymbol{x}^j)$ is
a function in $l^2 (\mathbf{G}^j_{h})$. Here, $\boldsymbol{V}\subset \mathbb{Z}^2$ is a finite index set. Because $L_h$ is formally diagonalized by the Fourier modes $\varphi(\boldsymbol{\theta},\boldsymbol{x}^j)= e^{i\boldsymbol{\theta}\cdot\boldsymbol{x}^j/\boldsymbol{h}}=e^{i\theta_1x^j_1/h}e^{i\theta_2x^j_2/h}$, where $\boldsymbol{\theta}=(\theta_1,\theta_2)$ and $i^2=-1$, we use $\varphi(\boldsymbol{\theta},\boldsymbol{x}^j)$ as a Fourier basis with $\boldsymbol{\theta}\in \big[-\frac{\pi}{2},\frac{3\pi}{2}\big)^{2}$. High and low frequencies for standard coarsening (as considered here) are given by
\begin{equation*}
  \boldsymbol{\theta}\in T^{{\rm low}} =\left[-\frac{\pi}{2},\frac{\pi}{2}\right)^{2}, \, \boldsymbol{\theta}\in T^{{\rm high}} =\displaystyle \left[-\frac{\pi}{2},\frac{3\pi}{2}\right)^{2} \bigg\backslash \left[-\frac{\pi}{2},\frac{\pi}{2}\right)^{2}.
\end{equation*}

\begin{definition}\label{formulation-symbol-FEM}
 If, for all  functions $\varphi(\boldsymbol{\theta},\boldsymbol{x}^j)$,
 \begin{equation*}
   L_{h}\varphi(\boldsymbol{\theta},\boldsymbol{x}^j)= \widetilde{L}_{h} (\boldsymbol{\theta})\varphi(\boldsymbol{\theta},\boldsymbol{x}^j),
 \end{equation*}
 we call $\widetilde{L}_{h}(\boldsymbol{\theta})=\displaystyle\sum_{\boldsymbol{\kappa}\in\boldsymbol{V}}s_{\boldsymbol{\kappa}}e^{i\boldsymbol{\theta}\boldsymbol{\kappa}}$ the symbol of $L_{h}$.
\end{definition}

In what follows, we consider $(3\times 3)$ linear systems of operators, which read
\begin{equation}\label{block-system-stencil}
  \mathcal{L}_h=\begin{pmatrix}
       L_h^{1,1}&         L_h^{1,2} &              L_h^{1,3}  \\
       L_h^{2,1}&         L_h^{2,2} &              L_h^{2,3} \\
       L_h^{3,1}&         L_h^{3,2} &              L_h^{3,3} \\
    \end{pmatrix}=\begin{pmatrix}
      -\Delta_{h} & 0 &  (\partial_{x})_{h}\\
      0 & -\Delta_{h} & (\partial_{y})_{h} \\
      -(\partial_{x})_{h}  & -(\partial_{y})_{h} & L_h^{3,3}
    \end{pmatrix},
\end{equation}
where $L^{3,3}_h$ depends on which discretization we use.

 For the stabilized $Q_1-Q_1$ approximations,  the degrees of freedom for both velocity and pressure are only located on $\mathbf{G}^{1}_h$ as pictured at left of Figure \ref{Q2-Q1-mesh}. In this setting, the $L_h^{k,\ell}(k,\ell=1,2,3)$ in (\ref{block-system-stencil}) are scalar Toeplitz operators. Denote $\widetilde{\mathcal{L}}_h$ as the symbol of $\mathcal{L}_h$. Each entry in $\widetilde{\mathcal{L}}_{h}$ is computed as the (scalar) symbol of the corresponding block of $L^{k,\ell}_{h}$, following Definition
 \ref{formulation-symbol-FEM}.  Thus, $\widetilde{\mathcal{L}}_h$ is a $3\times 3$ matrix. All blocks in $\mathcal{L}_h$ are diagonalized by the same
 transformation on a collocated mesh.

However, for the $Q_2-Q_1$ discretization, the degrees of freedom for velocity are located on $\mathbf{G}_h=\bigcup_{j=1}^4\mathbf{G}_h^{j}$, containing four types of meshpoints  as shown at right of Figure \ref{Q2-Q1-mesh}. The Laplace operator in (\ref{block-system-stencil}) is defined by extending (\ref{defi-symbol-FEM}), with $\boldsymbol{V}$ taken to be a finite index set of values, $\boldsymbol{V}=V_{N}\bigcup V_{X}\bigcup V_{Y}\bigcup V_{C}$ with $V_{N}\subset\mathbb{Z}^2$, $V_{X}\subset\big\{(z_x+\frac{1}{2},z_y)|(z_x,z_y)\in\mathbb{Z}^2\big\}$, $V_{Y}\subset\big\{(z_x,z_y+\frac{1}{2})|(z_x,z_y)\in\mathbb{Z}^2\big\}$, and $V_{C}\subset\big\{(z_x+\frac{1}{2},z_y+\frac{1}{2})|(z_x,z_y)\in\mathbb{Z}^2\bigg\}$. With this, the (scalar) $Q_2$ Laplace operator is naturally treated as  a block operator,  and the Fourier representation of each block can be calculated based on Definition \ref{formulation-symbol-FEM}, with the Fourier bases adapted to account for the staggering of the mesh points. Thus, the symbols of $L_h^{1,1}$ and  $L_h^{2,2}$  are  $4\times 4$ matrices. For more details of LFA for the Laplace operator using higher-order finite-element methods,  refer to \cite{HM2018LFALaplace}.  Similarly to the Laplace operator, both terms in the gradient, $(\partial_x)_h$ and $(\partial_y)_h$, can be treated as  ($4\times 1$)-block operators. Then, the symbols of $L^{1,3}_h$ and $L^{2,3}_h$ are $4\times 1$ matrices, calculated based on Definition \ref{formulation-symbol-FEM} adapted for the mesh staggering. The symbols of $L^{3,1}_h$ and $L^{3,2}_h$ are the conjugate transposes of those of $L^{1,3}_h$ and $L^{2,3}_h$, respectively. Finally, $L^{3,3}_h=0$. Accordingly, $\widetilde{\mathcal{L}}_{h}$ is a $9\times9$ matrix for the $Q_2-Q_1$ discretization.

\begin{definition}\label{err-prop-LFA-FEM}
 The error-propagation symbol, $\widetilde{\mathcal{S}}_{h}(\boldsymbol{\theta})$, for a block smoother $\mathcal{S}_{h}$ on the infinite grid  $\mathbf{G}_{h}$ satisfies
\begin{equation*}
  \mathcal{S}_{h}\varphi(\boldsymbol{\theta},\boldsymbol{x})=\widetilde{\mathcal{S}}_{h}\varphi(\boldsymbol{\theta},\boldsymbol{x}), \,\,\boldsymbol{\theta}\in \bigg[-\frac{\pi}{2},\frac{3\pi}{2}\bigg)^{2},
\end{equation*}
for all $\varphi(\boldsymbol{\theta},\boldsymbol{x})$, and the corresponding smoothing factor for $\mathcal{S}_{h}$ is given by
\begin{equation*}
  \mu_{{\rm loc}}=\mu_{{\rm loc}}(\mathcal{S}_{h})=\max_{\boldsymbol{\theta}\in T^{{
  \rm high}}}\big\{\big|\lambda(\widetilde{\mathcal{S}}_{h}(\boldsymbol{\theta}))\big| \,\,\big\},
\end{equation*}
where $\lambda$ is an eigenvalue of $\widetilde{\mathcal{S}}_{h}(\boldsymbol{\theta})$.
\end{definition}
In Definition \ref{err-prop-LFA-FEM}, $\mathbf{G}_{h}=\mathbf{G}^1_{h}$ for the stabilized case (and $\widetilde{\mathcal{S}}_{h}(\boldsymbol{\theta})$ is a $3\times 3$ matrix) and $\mathbf{G}_h=\bigcup_{j=1}^4\mathbf{G}_h^{j}$  for the stable case (where $\widetilde{\mathcal{S}}_{h}(\boldsymbol{\theta})$ is a $9\times 9$ matrix).

The error-propagation symbol for a relaxation scheme, represented by matrix $M_h$, applied to  either the stabilized or  stable scheme is written as
\begin{equation*}
\widetilde{ \mathcal{S}}_h(\boldsymbol{p},\omega,\boldsymbol{\theta})=I-\omega \widetilde{M}_h^{-1}(\boldsymbol{\theta})\widetilde{\mathcal{L}}_h(\boldsymbol{\theta}),
\end{equation*}
where $\boldsymbol{p}$ represents parameters within $M_h$, the block approximation to $\mathcal{L}_h$, $\omega$ is an overall weighting factor,
and $\widetilde{M}_h$ and $\widetilde{\mathcal{L}}_h$ are the symbols for $M_h$ and $\mathcal{L}_h$, respectively.  Note that $\mu_{\rm loc}$ is a function of some parameters in Definition \ref{err-prop-LFA-FEM}. In this paper, we  focus on minimizing $\mu_{\rm loc}$ with respect to these parameters, to obtain the optimal LFA smoothing factor.
\begin{definition}\label{def-opt-FEM}
  Let $\mathcal{D}$ be the set of allowable parameters and define the optimal smoothing factor over $\mathcal{D}$ as
  \begin{equation*}
    \mu_{{\rm opt}}=\min_{\mathcal{D}}{\mu_{{\rm loc}}}.
  \end{equation*}
\end{definition}

If the standard LFA assumption of an ``ideal'' CGC holds, then the two-grid convergence factor can be estimated by the smoothing factor, which is easy to compute. However, as expected, we will see that this idealized CGC does not lead to a good prediction for some cases we consider below. When the LFA smoothing factor fails to predict the true two-grid convergence factor, the LFA two-grid convergence factor can still be used. Thus, we give a brief introduction to the LFA two-grid convergence factor in the following.

Let
\begin{eqnarray*}
\boldsymbol{\alpha}&=&(\alpha_1,\alpha_2)\in\big\{(0,0),(1,0),(0,1),(1,1)\big\},\\
\boldsymbol{\theta}^{\boldsymbol{\alpha}}&=&(\theta_1^{\alpha_1},\theta_2^{\alpha_2})=\boldsymbol{\theta}+\pi\cdot\boldsymbol{\alpha},\,\,
\boldsymbol{\theta}:=\boldsymbol{\theta}^{00}\in T^{{\rm low}}.
\end{eqnarray*}
We use the ordering of $\boldsymbol{\alpha}=(0,0),(1,0),(0,1),(1,1)$ for the four harmonics. To apply LFA to the two-grid operator,
 \begin{equation}\label{TG-Operator-FEM}
    \boldsymbol{M}^{\rm TGM}_h= {\mathcal{ S}}^{\nu_2}_h\mathcal{M}^{{\rm CGC}}_h{\mathcal{ S}}^{\nu_1}_h,
 \end{equation}
we require the representation of the CGC operator,
 \begin{equation*}
   \mathcal{M}^{{\rm CGC}}_h=I-P_h ({\mathcal{L}}^{*}_{2h})^{-1}R_h \mathcal{L}_{h},
 \end{equation*}
 where  $P_h$ is the multigrid interpolation operator and $R_h$ is the restriction operator.  The coarse-grid operator, $\mathcal{L}^{*}_{2h}$, can be either the Galerkin or rediscretization operator.

 Inserting the representations of $\mathcal{S}_h, \mathcal{L}_h, \mathcal{L}^{*}_{2h}, P_h, R_h$ into (\ref{TG-Operator-FEM}), we obtain the Fourier representation of two-grid error-propagation operator as
\begin{equation*}
 \widetilde{ \boldsymbol{M}}^{\rm TGM}_h(\boldsymbol{\theta})= \widetilde{\boldsymbol {S}}^{\nu_2}_h(\boldsymbol{\theta})\big(I-\widetilde{\boldsymbol {P}}_h(\boldsymbol{\theta})(\widetilde{\mathcal{L}}^{*}_{2h}(2\boldsymbol{ \theta}))^{-1}\widetilde{\boldsymbol{ R}}_h(\boldsymbol{\theta})\widetilde{\boldsymbol{ L}}_{h}(\boldsymbol{\theta})\big)\widetilde{\boldsymbol{ S}}^{\nu_1}_h(\boldsymbol{\theta}),
\end{equation*}
where
\begin{eqnarray*}
\widetilde{\boldsymbol{L}}_h(\boldsymbol{\theta})&=&\text{diag}\left\{\widetilde{\mathcal{L}}_h(\boldsymbol{\theta}^{00}), \widetilde{\mathcal{L}}_h(\boldsymbol{\theta}^{10}),\widetilde{\mathcal{L}}_h(\boldsymbol{\theta}^{01}),
\widetilde{\mathcal{L}}_h(\boldsymbol{\theta}^{11})\right\},\\
\widetilde{\boldsymbol{S}}_h(\boldsymbol{\theta})&=&\text{diag}\left\{\widetilde{\mathcal{S}}_h(\boldsymbol{\theta}^{00}),
\widetilde{\mathcal{S}}_h(\boldsymbol{\theta}^{10}),\widetilde{\mathcal{S}}_h(\boldsymbol{\theta}^{01}),
\widetilde{\mathcal{S}}_h(\boldsymbol{\theta}^{11})\right\},\\
\widetilde{\boldsymbol{P}}_h(\boldsymbol{\theta})&=&\left(\widetilde{P}_h(\boldsymbol{\theta}^{00});\widetilde{P}_h(\boldsymbol{\theta}^{10});
\widetilde{P}_h(\boldsymbol{\theta}^{01});\widetilde{P}_h(\boldsymbol{\theta}^{11}) \right),\\
\widetilde{\boldsymbol{R}}_h(\boldsymbol{\theta})&=&\left(\widetilde{R}_h(\boldsymbol{\theta}^{00}),\widetilde{R}_h(\boldsymbol{\theta}^{10}),
\widetilde{R}_h(\boldsymbol{\theta}^{01}),\widetilde{R}_h(\boldsymbol{\theta}^{11}) \right),
\end{eqnarray*}
in which ${\rm diag}\{T_1,T_2,T_3,T_4\}$ stands for the block diagonal matrix with diagonal blocks, $T_1, T_2, T_3$, and $T_4$.

Here, we use the standard finite-element interpolation operators and their transposes for restriction. For $Q_1$, the symbol is well-known  \cite{MR1807961} while, for the nodal basis for $Q_2$, the symbol is given in \cite{HM2018LFALaplace}.
\begin{definition}
The asymptotic two-grid convergence factor, $\rho_{{\rm asp}}$, is defined as
\begin{equation*}\label{real-TGM}
  \rho_{{\rm asp}} = {\rm sup}\{\rho(\widetilde{\boldsymbol{M}}_h(\boldsymbol{\theta})^{{\rm TGM}}): \boldsymbol{\theta}\in T^{{\rm low}}\}.
\end{equation*}
\end{definition}
In what follows, we consider a discrete form of $\rho_{\rm asp}$, denoted by $\rho_{h}$, resulting from sampling $\rho_{\rm asp}$ over only a finite set of frequencies. For simplicity, we drop the subscript $h$ throughout the rest of this paper, unless necessary for clarity.
\subsection{Fourier representation of discretization operators}

\subsubsection{Fourier representation of the stabilized $Q_1-Q_1$ discretization}
 By standard calculation, the symbols of the $Q_1$ stiffness and mass stencils  are
\begin{eqnarray*}
  \widetilde{A}_{Q_1}(\theta_1,\theta_2)&=&\frac{2}{3}(4-\cos\theta_1-\cos\theta_2-2\cos\theta_1\cos\theta_2),\\
  \widetilde{M}_{Q_1}(\theta_1,\theta_2) &=&\frac{h^2}{9}(4+2\cos\theta_1+2\cos\theta_2+\cos\theta_1\cos\theta_2),
\end{eqnarray*}
respectively.  The stencils of the partial derivative operators $(\partial_{x})_{h}$ and $(\partial_{y})_{h}$ are
\begin{equation*}
  B_x^{T}=\frac{h}{12}\begin{bmatrix}
   -1 & 0 & 1\\
   -4 & 0 &4\\
   -1 & 0 & 1
   \end{bmatrix},\,\,  B_y^{T}=\frac{h}{12}\begin{bmatrix}
   1 & 4 & 1\\
   0 & 0 & 0\\
   -1& -4 &-1
   \end{bmatrix},
\end{equation*}
respectively, and the corresponding symbols are
\begin{equation*}
  \widetilde{B}_x^{T}(\theta_1,\theta_2)=\frac{ih}{3}\sin\theta_1(2+\cos\theta_2),\,\,\widetilde{B}_y^{T}(\theta_1,\theta_2)=\frac{ih}{3}(2+\cos\theta_1)\sin\theta_2,
\end{equation*}
where $T$ denotes the conjugate transpose.
Thus, the symbols of the stabilized finite-element discretizations of the Stokes equations are given by
 \begin{equation*}
   \widetilde{\mathcal{L}}(\theta_1,\theta_2)=\begin{pmatrix}
      \widetilde{ A}_{Q_1} & 0 &  \widetilde{B}_x^{T}\\
      0 & \widetilde{A}_{Q_1} & \widetilde{B}_y^{T} \\
      \widetilde{B}_x  & \widetilde{B}_y & \widetilde{L}_h^{3,3}
    \end{pmatrix}  :=\begin{pmatrix}
       a& 0 &  b_1\\
      0 &    a & b_2 \\
      - b_1    &  - b_2 & -c
    \end{pmatrix}.
  \end{equation*}
For the Poisson-stabilized discretization, the symbol of $-L_{h}^{3,3}$ is $c=c_1=a\beta h^2$.  For the projection stabilized method,
 following (\ref{projection-stabilized}),  the symbol of $-L_{h}^{3,3}$  is
 \begin{equation}\label{c2-form}
 c_2=\bigg(\frac{4+2\cos\theta_1+2\cos\theta_2+\cos\theta_1\cos\theta_2}{9}-\frac{(1+\cos\theta_1)(1+\cos\theta_2)}{4}\bigg)h^2.
 \end{equation}
 For convenience,  we write $-C$ for the last block of Equation (\ref{saddle-structure-FEM}), and its symbol as $-c$ in the rest of this paper.

\subsubsection{Fourier representation of stable $Q_2-Q_1$ discretizations}
The symbols of the stiffness and mass stencils for the $Q_2$ discretization using nodal basis functions in 1D are
\begin{equation*}
\widetilde{A}_{Q_2}(\theta) =\frac{1}{3h}\begin{pmatrix}
       14+2\cos\theta            &      -16\cos\frac{\theta}{2}  \\
        -16\cos\frac{\theta}{2}  &       16
    \end{pmatrix},\,\,\,
\widetilde{M}_{Q_2}(\theta)=\frac{h}{30}\begin{pmatrix}
       8-2\cos\theta            &      4\cos\frac{\theta}{2}  \\
        4\cos\frac{\theta}{2}  &       16
    \end{pmatrix},
\end{equation*}
respectively \cite{HM2018LFALaplace}. Here, we note that the $(1,1)$  entries correspond to the symbols associated with basis functions at the nodes of the mesh, while the $(2,2)$  entries correspond to the symbols associated with cell-centre (bubble) basis functions. The off-diagonal entries express the interaction between the two types of basis functions.  Then, the Fourier representation of $ -\Delta_{h}$ in 2D can be written as a tensor product,
\begin{eqnarray}
  \widetilde{A}_{2}(\theta_1,\theta_2) &=&\widetilde{A}_{Q_2}(\theta_2)\otimes \widetilde{M}_{Q_2}(\theta_1)+
  \widetilde{M}_{Q_2}(\theta_2)\otimes\widetilde{A}_{Q_2}(\theta_1). \nonumber
\end{eqnarray}
The tensor product preserves block structuring; that is, $\widetilde{A}_{2}(\theta_1,\theta_2)$ is a $4\times 4$ matrix, ordered as mesh nodes, $x$-edge midpoints, $y$-edge midpoints, and cell centres. Each row of $\widetilde{A}_{2}(\theta_1,\theta_2)$ reflects the connections between one  of the four types of degrees of freedom with each of these four types. Similarly, there are four types of stencils for $(\partial_{x})_h$ and $(\partial_{y})_h$.

 The stencils and the symbols of $(\partial_{x})_h$ for the nodal, $x$-edge, $y$-edge, and cell-centre degrees of freedom are
\begin{eqnarray*}
  B_{N}&=&\frac{h}{18}\begin{bmatrix}
   0 & 0 & 0\\
   -1 & 0 &1\\
   0 & 0 & 0
   \end{bmatrix},\,\,
  \widetilde{B}_{N}(\theta_1,\theta_2)=\frac{ih}{9}\sin\theta_1,\\
 B_{X}&=&\frac{h}{18}\begin{bmatrix}
   0  & 0\\
   -4  & 4\\
   0   & 0
   \end{bmatrix},\,\,\,\,\quad
  \widetilde{B}_{X}(\theta_1,\theta_2)=\frac{2ih}{9}\sin\frac{\theta_1}{2},\\
  B_{Y}&=&\frac{h}{18}\begin{bmatrix}
   -1 & 0 & 1\\
   -1 & 0 &1\\
   \end{bmatrix},\,\,
  \widetilde{B}_{Y}(\theta_1,\theta_2)=\frac{2ih}{9}\sin\theta_1\cos\frac{\theta_2}{2},\\
  B_{C}&=&\frac{h}{18}\begin{bmatrix}
   -4 & 4\\
   -4 & 4
   \end{bmatrix},\,\,\,\,\quad
  \widetilde{B}_{C}(\theta_1,\theta_2)=\frac{8ih}{9}\sin\frac{\theta_1}{2}\cos\frac{\theta_2}{2},
\end{eqnarray*}
respectively. Denote $\widetilde{B}_{Q_2,x}(\theta_1,\theta_2)^{T} =[\widetilde {B}_{N};\widetilde{ B}_{X};\widetilde{ B}_{Y};\widetilde{ B}_{C}]$.

Similarly to $\widetilde{B}_{Q_2,x}(\theta_1,\theta_2)^T$, the symbol of the stencil of $(\partial_{y})_h$ can be written as
\begin{equation*}
\widetilde{B}_{Q_2,y}(\theta_1,\theta_2)^{T} =
 [\widetilde {B}_{N}(\theta_2,\theta_1);\widetilde {B}_{Y}(\theta_2,\theta_1);\widetilde{ B}_{X}(\theta_2,\theta_1);
 \widetilde{ B}_{C}(\theta_2,\theta_1)].
 \end{equation*}
Thus, the Fourier representation of the $Q_2-Q_1$ finite-element discretization of the Stokes equations can be written as
 \begin{equation}\label{Full-symbol}
   \widetilde{\mathcal{L}}_{h}(\theta_1,\theta_2) =\begin{pmatrix}
      \widetilde{A}_{2}(\theta_1,\theta_2) & 0 &  \widetilde {B}_{Q_2,x}(\theta_1,\theta_2)^{T}\\
      0 & \widetilde{A}_{2}(\theta_1,\theta_2) & \widetilde {B}_{Q_2,y}(\theta_1,\theta_2)^{T} \\
      \widetilde {B}_{Q_2,x}(\theta_1,\theta_2)  & \widetilde {B}_{Q_2,y}(\theta_1,\theta_2) &0
    \end{pmatrix}.
  \end{equation}
Note that the Fourier symbol for the $Q_2-Q_1$ discretization is a $9\times 9$ matrix, and that the LFA smoothing factor for the $Q_2$ approximation generally fails to predict the true two-grid convergence factor \cite{HM2018LFALaplace,MR2840198}. The same behavior is seen for the relaxation schemes considered here. Therefore, we do not present smoothing factor analysis for this case and only optimize two-grid LFA predictions numerically.

\section{Relaxation for  $Q_1-Q_1$ discretizations}\label{LFA-relaxation-FEM}
\subsection{DWJ relaxation}

Distributive GS (DGS) relaxation \cite{MR558216,oosterlee2006multigrid} is well known to be highly efficient for the MAC finite-difference discretization \cite{MR1807961},  and other discretizations \cite{MR3427891,chen2015multigrid}. Its sequential nature is often seen as a significant drawback. However, Distributive weighted Jacobi (DWJ) relaxation was recently shown to achieve  good performance for the MAC discretization \cite{NLA2147}. Thus, we consider DWJ relaxation for the finite-element discretizations considered here. The discretized distribution operator can be represented by the preconditioner
 \begin{equation*}
   \mathcal{P} =\begin{pmatrix}
      I_{h} & 0 &  (\partial_{x})_{h}\\
      0 & I_{h} & (\partial_{y})_{h} \\
      0  &0  & \Delta_{h}
    \end{pmatrix}.
  \end{equation*}
Then, we apply blockwise weighted-Jacobi relaxation to the distributed operator
   \begin{equation}\label{DWJ-system-FEM}
     \mathcal{L}\mathcal{P}\approx \mathcal{L}^{*} =\begin{pmatrix}
      -\Delta_{h} & 0 &  0\\
      0 & -\Delta_{h} & 0 \\
      -(\partial_{x})_{h}  & -(\partial_{y})_{h}  & -(\partial_{x})_{h}^2-(\partial_{y})_{h}^2+L^{3,3}\Delta_{h}
    \end{pmatrix},
  \end{equation}
where we note that the operators $(\partial_x)_h^2$ and $(\partial_y)_h^2$ are formed by taking products of the discrete derivative operators and, thus, do not satisfy the identity $(\partial_x)_h^2+(\partial_y)_h^2=\Delta_h$.

The discrete matrix form of $\mathcal{P}$ is
\begin{equation*}\label{P-Precondtion}
   \mathcal{P}=  \begin{pmatrix}
      I & B^{T}\\
      0 & -A_p\\
    \end{pmatrix},
\end{equation*}
where $A_p$ is the Laplacian operator discretized at the pressure points. For standard distributive weighted-Jacobi relaxation (with weights $\alpha_1,\alpha_2$), we need to solve a system of the form
\begin{equation}\label{DWJ-Precondition-FEM}
   M_{D}\delta \hat{x}=  \begin{pmatrix}
      \alpha_1 {\rm diag}(A) & 0\\
      B & \alpha_2 h^2 I\\
    \end{pmatrix}
    \begin{pmatrix} \delta \mathcal{\hat{U}} \\ \delta \hat{p}\end{pmatrix}
  =\begin{pmatrix} r_{\mathcal{U}} \\ r_{p}\end{pmatrix},
\end{equation}
then distribute the updates as $\delta x=\mathcal{P}\delta \hat{x}$.  We use $h^2$ in the $(2,2)$ block of (\ref{DWJ-Precondition-FEM}), because the diagonal entries of the $(2,2)$ block will be of the form of a constant times $h^2$ (up to boundary conditions), for both stabilization terms.
The error propagation operator for the scheme is, then, $I-\omega\mathcal{P}M_{D}^{-1}\mathcal{L}$.

The symbol of the blockwise weighted-Jacobi  operator, $M_D$, is
  \begin{equation*}
   \widetilde{M}_{D}(\theta_1,\theta_2) =\begin{pmatrix}
       \frac{8}{3}\alpha_1 & 0 & 0  \\
      0 &  \frac{8}{3}\alpha_1  & 0 \\
      -b_1   & -b_2 & h^2\alpha_2
    \end{pmatrix}.
\end{equation*}
By standard calculation, the eigenvalues of the error-propagation symbol, $\mathcal{\widetilde{S}}_{D}(\alpha_1,\alpha_2, \omega,\boldsymbol{\theta})=I- \omega\widetilde{\mathcal{P}} \widetilde {M}_{D}^{-1}\widetilde{\mathcal{L}}$, are
\begin{equation}\label{eig3-DWJ}
  1-\frac{\omega}{\alpha_1}y_1,\,\,\, 1-\frac{\omega}{\alpha_1}y_1, 1-\frac{\omega}{\alpha_2}y_2,
\end{equation}
 where $y_1=\frac{3a}{8}$ and $y_2=\frac{-b_1^2-b_2^2+ac}{h^2}.$

Noting that $y_1=\frac{3a}{8}$ is very simple, we first consider a lower bound on the optimal LFA smoothing factor corresponding to $y_1$.
 \begin{lemma}\label{common-eig}
 \begin{equation*}
  \mu^{*}:= \min_{(\alpha_1,\omega)}\max_{\boldsymbol{\theta}\in T^{{\rm high}}}\bigg\{\big|1-\frac{ \omega}{\alpha_1}y_1\big|\bigg\}=\frac{1}{3},
 \end{equation*}
 and this value is achieved if and only if $\frac{\omega}{\alpha_1}=\frac{8}{9}$.
 \end{lemma}
 \begin{proof}
 It is easy to check that $a= \frac{2(4-\cos\theta_1-\cos\theta_2-2\cos\theta_1\cos\theta_2)}{3}\in [2,4]$ for $\boldsymbol{\theta}\in T^{{\rm high}}$.
  The minimum of $y_1$ is $y_{1,\rm{min}}=\frac{3}{4}$ with $(\cos\theta_1,\cos\theta_2)=(0,1)$ or $(1,0)$ and the maximum is
   $y_{1,\rm{max}}=\frac{3}{2}$ with $(\cos\theta_1,\cos\theta_2)=(1,-1)$ or $(-1,1)$. Thus,
   $\mu^{*}=\frac{y_{1,\rm{max}}+y_{1,\rm{min}}}{y_{1,\rm{max}}-y_{1,\rm{min}}}=\frac{1}{3}$ under the condition
   $\frac{\omega}{\alpha_1}=\frac{2}{y_{1,\rm{min}}+y_{1,\rm{max}}}= \frac{8}{9}$.
 \end{proof}
  \begin{remark}
   The optimal smoothing factor for damped Jacobi relaxation for the $Q_1$ finite-element discretization of
   the Laplacian is $\frac{1}{3}$ with $ \frac{\omega}{\alpha}=\frac{8}{9}$. Thus, this offers an intuitive lower bound on the possible performance of block relaxation schemes that include this as a piece of the overall relaxation.
 \end{remark}

 From (\ref{eig3-DWJ}), we see that the only difference between the eigenvalues of DWJ relaxation for the Poisson-stabilized and projection stabilized methods is in the third eigenvalue, which depends on $y_2$ and, consequently, on the stabilization term.

\subsubsection{Poisson-stabilized discretization with DWJ relaxation}
For the Poisson-stabilized case, $y_2=\frac{-b_1^2-b_2^2+ac}{h^2}$ with $c=\beta \alpha h^2$ and $\beta=\frac{1}{24}$. By standard calculation, $y_{2,{\rm{min}}}=\frac{8}{27}$, with $\big(\cos\theta_1,\cos\theta_2\big)=(-1,-1)$, and  $y_{2,\rm{max}}=\frac{64}{51}$ with
 $\big(\cos\theta_1,\cos\theta_2\big)=(\frac{8}{17},0)$ or $(0,\frac{8}{17})$ .

\begin{theorem}\label{optimal-PoSD-DWJ}
  The optimal smoothing factor for the Poisson-stabilized discretization with DWJ relaxation  is $\frac{55}{89}$, that is,
   \begin{equation*}
   \mu_{{\rm opt}}=\displaystyle \min_{(\alpha_{1}, \omega,\alpha_2)}\max_{ \boldsymbol{\theta}\in T^{{\rm high}}}\bigg\{\big|\lambda(\mathcal{\widetilde{S}}_D(\alpha_{1}, \alpha_2,\omega, \boldsymbol{\theta}))\big|\bigg\}=\frac{55}{89}\approx 0.618,
\end{equation*}
and is achieved if and only if
\begin{equation}\label{beta-DWJ-Parameter-domain}
   \frac{\omega}{\alpha_2}=\frac{459}{356},\,\,\frac{136}{267}\leq\frac{\omega}{\alpha_1}\leq\frac{96}{89}.
\end{equation}
\end{theorem}
\begin{proof}
$\displaystyle \min_{(\alpha_2,\omega)}\max_{\boldsymbol{\theta}\in T^{{\rm high}}}\bigg\{\big|1-\frac{ \omega}{\alpha_2}y_2\big|\bigg\}= \frac{y_{2,\rm{max}}-y_{2,\rm{min}}}{y_{2,\rm{max}}+y_{2,\rm{min}}}=\frac{55}{89}$ with the condition that $\frac{\omega}{\alpha_2}=\frac{2}{y_{2,\rm{max}}+y_{2,\rm{min}}}=\frac{459}{356}$. Because $\frac{55}{89}>\frac{1}{3}$, we need to require $|1-\frac{\omega}{\alpha_1}y_1|\leq \frac{55}{89}$ for all $y_1$ to achieve this factor. It follows that  $\frac{136}{267}\leq \frac{\omega}{\alpha_1}\leq\frac{96}{89}$.
\end{proof}

\subsubsection{Projection stabilized discretization with DWJ relaxation}
For the projection stabilized discretization, $y_2$ depends on $c_2$ given in (\ref{c2-form}), and standard calculation gives $y_{2,{\rm{min}}}=\frac{8}{27}$ with  $\big(\cos\theta_1,\cos\theta_2\big)=(-1,-1)$ and $y_{2,\rm{max}}=\frac{3}{2}$ with $(\cos\theta_1,\cos\theta_2)=(-\frac{1}{2},1)$ or $(1,-\frac{1}{2})$.
\begin{theorem}\label{optimal-PrSD-DWJ}
  The optimal smoothing factor for the projection stabilized  discretization  with DWJ relaxation  is $\frac{65}{97}$,
   that is,
\begin{equation*}
   \mu_{{\rm opt}}=\displaystyle \min_{(\alpha_{1}, \omega,\alpha_2)}\max_{ \boldsymbol{\theta}\in T^{{\rm high}}}\bigg\{\big|\lambda(\mathcal{\widetilde{S}}_D(\alpha_{1}, \alpha_2,\omega,\boldsymbol{\theta}))\big|\bigg\}=\frac{65}{97}\approx 0.670,
\end{equation*}
and is achieved if and only if
\begin{equation}\label{D-DWJ-Parameter-domain}
  \frac{\omega}{\alpha_2}=\frac{108}{97},\,\,\frac{128}{291}\leq\frac{\omega}{\alpha_1}\leq\frac{108}{97}.
\end{equation}
\end{theorem}

\begin{proof}
$\displaystyle \min_{(\alpha_2,\omega)}\max_{\boldsymbol{\theta}\in T^{{\rm high}}}\bigg\{\big|1-\frac{ \omega}{\alpha_2}y_2\big|\bigg\}= \frac{y_{2,\rm{max}}-y_{2,\rm{min}}}{y_{2,\rm{max}}+y_{2,\rm{min}}}=\frac{65}{97}$  with the condition that $\frac{\omega}{\alpha_2}=\frac{2}{y_{2,\rm{max}}+y_{2,\rm{min}}}=\frac{108}{97}$. Since $\frac{65}{97}>\frac{1}{3}$, we need to require $|1-\frac{\omega}{\alpha_1}y_1|\leq \frac{65}{97}$ for all $y_1$ to achieve this factor, which leads to  $\frac{128}{291}\leq \frac{\omega}{\alpha_1}\leq\frac{108}{97}$.
\end{proof}

Comparing the Poisson-stabilized and projection stabilized discretizations using DWJ, we see that the optimal LFA smoothing factor for the Poisson-stabilized discretization slightly outperforms that of the projection stabilized discretization. In both cases, a stronger relaxation on the $(3,3)$ block of (\ref{DWJ-system-FEM}) would be needed in order to improve performance to match the lower bound on the convergence factor of $\frac{1}{3}$. A natural approach is to using more iterations to solve the pressure equation in DWJ. We explore the LFA predictions for this case in the following.

\subsubsection{Stabilized discretization with 2 sweeps of Jacobi for DWJ relaxation}
Denote the $(3,3)$ block  of (\ref{DWJ-system-FEM}) as $G$. We consider applying two sweeps of weighted-Jacobi relaxation with equal weights, $\omega_J$, on the pressure equation. As before, we note that $G$ has a constant diagonal entry proportional to $h^2$, so we write weighted Jacobi relaxation on $G$ as $I-G_J^{-1}G$ for $G_J=\frac{h^2}{\omega_J}I$.   Thus, we can represent this relaxation scheme as solving
\begin{equation}\label{DWJ-Precondition-2Sweeps}
   M_{D,J}\delta \hat{x}=  \begin{pmatrix}
      \alpha_1 {\rm diag}(A) & 0\\
      B & \hat{G}\\
    \end{pmatrix}
    \begin{pmatrix} \delta \mathcal{\hat{U}} \\ \delta \hat{p}\end{pmatrix}
  =\begin{pmatrix} r_{\mathcal{U}} \\ r_{p}\end{pmatrix},
\end{equation}
where $\hat{G}=\Big(2G_J^{-1}-G_J^{-1}GG_J^{-1}\Big)^{-1}$.
The symbol of $M_{D,J}$, is
  \begin{equation*}
   \widetilde{M}_{D,J}(\theta_1,\theta_2) =\begin{pmatrix}
       \frac{8}{3}\alpha_1 & 0 & 0  \\
      0 &  \frac{8}{3}\alpha_1  & 0 \\
      -b_1   & -b_2 & \frac{h^2}{2\omega_J-\omega_J^2y_2}
    \end{pmatrix}.
\end{equation*}
By standard calculation, the eigenvalues of the error-propagation symbol, $\mathcal{\widetilde{S}}_{D,J}(\alpha_1,\omega_J, \omega,\boldsymbol{\theta})=I- \omega\widetilde{\mathcal{P}} \widetilde {M}_{D,J}^{-1}\widetilde{\mathcal{L}}$, are
\begin{equation}\label{eig3-DWJ-2Jacobi}
  1-\frac{\omega}{\alpha_1}y_1,\,\,\, 1-\frac{\omega}{\alpha_1}y_1, 1- \omega y_3,
\end{equation}
where $y_3=\omega_Jy_2(2-\omega_Jy_2)$, where the symbol of $G$ is $h^2y_2$, with $y_2$ defined as in (\ref{eig3-DWJ}).  Note that $\mathcal{\widetilde{S}}_{D,J}$ has the same eigenvalue, $1-\frac{\omega}{\alpha_1}y_1$ as that of $\mathcal{\widetilde{S}}_{D}$. A natural question is whether  $\displaystyle\min_{(\alpha_1,\omega_J,\omega)}\max_{\boldsymbol{\theta}\in T^{{\rm high}}}\bigg\{\big|1-\omega y_3\big|\bigg\}=\frac{1}{3}$, which is shown in the following theorems.
 \begin{theorem}\label{optimal-PoSD-DWJ-2Jacobi}
  The optimal smoothing factor for the Poisson-stabilized  discretization  with 2 sweeps of Jacobi for DWJ relaxation  is $\frac{1}{3}$,
   that is,
\begin{equation*}
   \mu_{{\rm opt}}=\displaystyle \min_{(\alpha_{1}, \omega_J,\omega)}\max_{ \boldsymbol{\theta}\in T^{{\rm high}}}\bigg\{\big|\lambda(\mathcal{\widetilde{S}}_{D,J}(\alpha_{1}, \omega_J,\omega,\boldsymbol{\theta}))\big|\bigg\}=\frac{1}{3},
\end{equation*}
and is achieved if and only if $\frac{\omega}{\alpha_1}=\frac{8}{9}$ and either
\begin{eqnarray*}
   \frac{459}{356}\leq &\omega_J& \leq \frac{51}{64}(1+\frac{\sqrt{2}}{2}),\\
  \frac{2}{3\Big(\frac{64}{51}\omega_J(2-\frac{64}{51}\omega_J)\Big)}\leq &\omega& \leq \frac{4}{3},
\end{eqnarray*}
or
\begin{eqnarray*}
\frac{27}{8}(1-\frac{\sqrt{2}}{2}) \leq&\omega_J&\leq \frac{459}{356},\\
\frac{2}{3\Big(\frac{8}{27}\omega_J(2-\frac{8}{27}\omega_J)\Big)}\leq &\omega&\leq \frac{4}{3}.
\end{eqnarray*}
\end{theorem}

\begin{proof}
Recall that $y_3= \omega_Jy_2(2-\omega_Jy_2): =\xi(2-\xi)$, where $\xi=\omega_Jy_2$. Let
\begin{equation}\label{third-eig-2Jacobi}
  \mu^{**} =\min_{(\alpha_1,\omega_J,\omega)}\max_{\boldsymbol{\theta}\in T^{{\rm high}}}\bigg\{\big|1-\omega y_3\big|\bigg\}.
\end{equation}
We first show that $\mu^{**}\leq\frac{1}{3}$ under some conditions on the parameters, $\omega_J$ and $\omega$. Let $y_{3,{\rm min}}$ and $y_{3,{\rm max}}$ be the maximum and minimum of $y_3(\xi)=\xi(2-\xi)$, respectively. If $\mu^{**}\leq\frac{1}{3}$, then it must be that
\begin{equation}\label{inequal-compare-eig}
\frac{2}{3y_{3,{\rm min}}}\leq \omega\leq\frac{4}{3y_{3,{\rm max}}}.
\end{equation}
Next, we need to find what $y_{3,\min}$ and $y_{3,\max}$ are. As discussed earlier, $y_2\in[\frac{8}{27},\frac{64}{51}]$. Thus, $\xi\in [\frac{8}{27}\omega_J,\frac{64}{51}\omega_J]$, where $\omega_J>0$. Note that $y_3(\xi)=\xi(2-\xi)=-(\xi-1)^2+1$ is a quadratic function with the axis of symmetric, $\xi=1$. Thus, the extreme values of $y_{3}(\xi)$ are achieved at the points $\frac{8}{27}\omega_J$, $\frac{64}{51}\omega_J$ or 1.  Based on $\frac{64}{51}\omega_J\leq 1$ and $\frac{64}{51}\omega_J\geq 1$, we consider two cases.
\begin{enumerate}
\item If $\frac{64}{51}\omega_J\leq 1$, we have
\begin{equation}\label{compare-extrem-value}
y_{3,{\rm min}} =\frac{8}{27}\omega_J\Bigg(2-\frac{8}{27}\omega_J\Bigg), \,\, y_{3,{\rm max}}=\frac{64}{51}\omega_J\Bigg(2-\frac{64}{51}\omega_J\Bigg).
\end{equation}
Note that (\ref{inequal-compare-eig}) indicates that $y_{3,\max}\leq 2 y_{3,\min}$.  Combining with (\ref{compare-extrem-value}) leads to  $\omega_J\geq \frac{373}{394}$. However, $\omega_J \leq \frac{51}{64}<\frac{373}{394}$. Thus, there is no $\omega_J$ such that $\mu^{**}\leq \frac{1}{3}$ in this case.
\item To guarantee that $|1-\omega y_3|=|1-\omega\xi(2-\xi)|<1$, we require that $0<\xi<2$.  Assume that $1\leq \frac{64}{51}\omega_J<2$. It follows that $\frac{8}{27}\omega_J<1\leq\frac{64}{51}\omega_J$. Recall that $y_3(\xi)=\xi(2-\xi)=-(\xi-1)^2+1$.
\begin{itemize}
\item If $(\frac{64}{51}\omega_J-1)\geq (1-\frac{8}{27}\omega_J)$, we have
\begin{equation}\label{conditionA}
   \frac{459}{356}\leq \omega_J< \frac{51}{32}.
\end{equation}
Then, the extreme values of $y_3(\xi)$ are
\begin{equation}\label{compare-extrem-value2}
 y_{3,{\rm min}}=\frac{64}{51}\omega_J\Bigg(2-\frac{64}{51}\omega_J\Bigg), \,\,y_{3,{\rm max}}=y_3(1)=1.
\end{equation}
Substituting (\ref{compare-extrem-value2}) in to (\ref{inequal-compare-eig}), we have
\begin{equation}\label{Final-CondiA}
\frac{2}{3\Big(\frac{64}{51}\omega_J(2-\frac{64}{51}\omega_J)\Big)}\leq \omega \leq \frac{4}{3}.
\end{equation}
To guarantee (\ref{Final-CondiA}) makes sense, in combination with (\ref{conditionA}) gives
\begin{equation}\label{resultA}
  \frac{459}{356}\leq\omega_J\leq \frac{51}{64}(1+\frac{\sqrt{2}}{2}).
\end{equation}
Recall that there is another eigenvalue, $1-\frac{\omega}{\alpha_1}y_1$, of $\mathcal{\widetilde{S}}_{D,J}$. In order to obtain $\mu_{\rm opt}=\frac{1}{3}$, we thus require
\begin{eqnarray*}
   \frac{459}{356}\leq &\omega_J& \leq \frac{51}{64}(1+\frac{\sqrt{2}}{2}),\\
  \frac{2}{3\Big(\frac{64}{51}\omega_J(2-\frac{64}{51}\omega_J)\Big)}\leq &\omega& \leq \frac{4}{3},\\
  \frac{\omega}{\alpha_1}&=&\frac{8}{9}.
\end{eqnarray*}
\item  A similar argument holds if $(\frac{64}{51}\omega_J-1)\leq (1-\frac{8}{27}\omega_J)$, leading to the second set of conditions.
\end{itemize}
\end{enumerate}
\end{proof}

Note that the set of  parameters values defined in Theorem \ref{optimal-PoSD-DWJ-2Jacobi} is not empty, with parameters $\alpha_1 = \frac{3}{2},\omega =\frac{4}{3}$ and $\omega_J=1$ in the set.

\begin{theorem}\label{optimal-PrSD-DWJ-2Jacobi}
  The optimal smoothing factor for the projection stabilized  discretization  with two sweeps of Jacobi for DWJ relaxation  is $\frac{1}{3}$,
   that is,
\begin{equation*}
   \mu_{{\rm opt}}=\displaystyle \min_{(\alpha_{1}, \omega,\alpha_2)}\max_{ \boldsymbol{\theta}\in T^{{\rm high}}}\bigg\{\big|\lambda(\mathcal{\widetilde{S}}(\alpha_{1}, \alpha_2,\omega,\boldsymbol{\theta}))\big|\bigg\}=\frac{1}{3},
\end{equation*}
and is achieved if and only if $\frac{\omega}{\alpha_1}=\frac{8}{9}$ and either
\begin{eqnarray*}
   \frac{108}{97}\leq &\omega_J& \leq \frac{2}{3}(1+\frac{\sqrt{2}}{2}),\\
  \frac{2}{3\Big(\frac{3}{2}\omega_J(2-\frac{3}{2}\omega_J)\Big)}\leq &\omega& \leq \frac{4}{3},
\end{eqnarray*}
or
\begin{eqnarray*}
\frac{27}{8}(1-\frac{\sqrt{2}}{2}) \leq&\omega_J&\leq \frac{108}{97},\\
\frac{2}{3\Big(\frac{8}{27}\omega_J(2-\frac{8}{27}\omega_J)\Big)}\leq &\omega&\leq \frac{4}{3}.
\end{eqnarray*}
\end{theorem}
 \begin{proof}
   The proof is similar to that of Theorem \ref{optimal-PoSD-DWJ-2Jacobi}.
 \end{proof}
\begin{remark}
  Theorems \ref{optimal-PoSD-DWJ-2Jacobi} and \ref{optimal-PrSD-DWJ-2Jacobi} tell us that  two sweeps of weighted-Jacobi relaxation  on the pressure equation in DWJ are required to achieve optimal performance. This is different than the case of DWJ for the MAC discretization \cite{NLA2147}, where the optimal convergence factor of $\frac{3}{5}$ is attained with one sweep of relaxation on the pressure equation.
\end{remark}
\begin{remark}
Red-black Gauss-Seidel relaxation \cite{MR1807961} is an attractive tool for parallel computation as it typically offers better relaxation properties while retaining parallelism. However, due to the added coupling of the finite-element operators considered here, four-colour or nine-colour relaxation would be needed to decouple the updates. Thus, we restrict ourselves to weighted Jacobi relaxation.
\end{remark}
\subsection{Braess-Sarazin relaxation}
\label{sec:BSR}
Although DWJ relaxation is efficient, we see clearly in the above that it ``underperforms'' in relation to weighted Jacobi relaxation for the scalar Poisson problem unless additional work is done on the pressure equation. Furthermore,  proper construction of the preconditioner, $\mathcal{P}$, is not always possible or straightforward, especially for other types of saddle-point problems. Considering these obstacles, we also analyse other block-structured relaxation schemes. Braess-Sarazin-type algorithms were originally developed as a relaxation scheme for the Stokes equations \cite{braess1997efficient}, requiring the solution of a greatly simplified but global saddle-point system.
The (exact) BSR approach was first introduced in \cite{braess1997efficient}, where it was shown that a multigrid convergence rate of $O(k^{-1})$ can be achieved, where $k$ denotes the number of smoothing steps on each level.  As a relaxation scheme for the system in (\ref{saddle-structure-FEM}), one solves a system of the form
\begin{equation}\label{Precondtion-FEM}
   M_{E}\delta_x=  \begin{pmatrix}
      \alpha D & B^{T}\\
     B & -C\\
    \end{pmatrix}
    \begin{pmatrix} \delta \mathcal{U} \\ \delta p\end{pmatrix}
  =\begin{pmatrix} r_{\mathcal{U}} \\ r_{p}\end{pmatrix},
\end{equation}
where $D$ is an approximation to $A$, the inverse of which is easy to apply, for example $I, \textrm{or}\,\,{\rm diag}(A)$. Solutions of (\ref{Precondtion-FEM}) are computed in two stages as
\begin{eqnarray}
  S\delta p&=&\frac{1}{\alpha} BD^{-1}r_{\mathcal{U}}- r_{p}, \label{schur-solution-of-precondtion}\\
  \delta \mathcal{U}&=&\frac{1}{\alpha}D^{-1}(r_{\mathcal{U}}-B^{T}\delta p),\nonumber
\end{eqnarray}
where $S=\frac{1}{\alpha}BD^{-1}B^{T}+C$, and  $\alpha>0$ is a chosen weight for $D$ to obtain a better approximation to $A$.
We consider an additional  weight, $\omega$, for the global  update, $\delta x$, to improve the effectiveness of the correction to both the velocity and pressure unknowns.

There is a significant difficulty in practical use of exact BSR because it requires an exact inversion of the approximate Schur complement, $S$, which is typically very expensive. A broader class of iterative methods for the Stokes problem is discussed in \cite{MR1810326}, which demonstrated that the same $O(k^{-1})$ performance can be achieved as with exact BSR when the pressure correction equation is not solved exactly. In practice, an approximate solve is sufficient for the Schur complement system, such as with a few sweeps of weighted Jacobi relaxation or a few multigrid cycles. In what follows,
we take $D={\rm diag}(A)$ and  analyze exact BSR; to see what convergence factor can be achieved. In numerical experiments, we then consider whether it is possible to achieve the same convergence factor using an inexact solver. Note that some studies \cite{adler2016monolithic,MR3639325,adler2016constrained} have shown the efficiency of inexact Braess-Sarazin relaxation.
The symbol of $M_{E}$ is given by
\begin{equation*}
   \widetilde{M}_{E}(\theta_1,\theta_2) =\begin{pmatrix}
       \frac{8}{3}\alpha & 0 & b_1  \\
      0 &   \frac{8}{3}\alpha & b_2 \\
      -b_1   &  -b_2 &-c
    \end{pmatrix}.
\end{equation*}
The symbol of the error-propagation matrix for weighted exact BSR is $\mathcal{\widetilde{S}}_{E}(\alpha, \omega,\boldsymbol{\theta})=I- \omega \widetilde {M}_{E}^{-1}\widetilde{\mathcal{L}}$.
A standard calculation shows that the determinant of $\widetilde{\mathcal{L}}-\lambda \widetilde{M}_{E}$ is
\begin{eqnarray}\label{3-eigs-BSR}
 \pi_{E}(\lambda;\alpha) = (1-\lambda)(a-\frac{8}{3}\alpha \lambda)\bigg[(1-\lambda)(b_1^2+b_2^2)+(\frac{8}{3}\alpha \lambda-a)c\bigg].
\end{eqnarray}
We first establish a lower bound on the LFA smoothing factor for the stabilized method with BSR.
\begin{theorem}
The optimal LFA smoothing factor for the  Poisson-stabilized and projection stabilized discretizations with exact BSR is not less than $\frac{1}{3}$.
\end{theorem}
\begin{proof}
From (\ref{3-eigs-BSR}), two eigenvalues of $\widetilde { M}_{E}^{-1}\mathcal{\widetilde{ L}}$ are given by
\begin{equation*}
\lambda_1=1,\,\, \lambda_2=\frac{3a}{8\alpha},
\end{equation*}
which are independent of the stabilization term, $c$. From Lemma \ref{common-eig}, we know that for $\lambda_2$, the optimal smoothing factor is $\frac{1}{3}$, under the condition that $\frac{\omega}{\alpha}=\frac{8}{9}$. Note that if $|1-\omega\lambda_1|\leq \frac{1}{3}$,
   then $\frac{2}{3}\leq\omega\leq \frac{4}{3}$. Because there is another eigenvalue, $\lambda_3$, the optimal LFA smoothing factor is not less than $\frac{1}{3}$.
\end{proof}
Similarly to DWJ, we see that the Jacobi relaxation for the Laplacian discretization places a limit on the overall performance of BSR.
From (\ref{3-eigs-BSR}), the third eigenvalue of $\widetilde {M}_{E}^{-1}\widetilde{\mathcal{L}}$ is $\lambda_3=\frac{ac+b}{\frac{8}{3}\alpha c+b}$,  where $b=-(b_1^2+b_2^2)\geq 0$ (because both $b_1$ and $b_2$ are imaginary).  Thus,  we only need to check whether we can choose $\alpha$ and $\omega$ so that $|1-\omega\lambda_3|\leq\frac{1}{3}$ over all high frequencies, while also ensuring $|1-\omega\lambda_1|\leq\frac{1}{3}$ and $|1-\omega\lambda_2|\leq\frac{1}{3}$ .

\begin{theorem}\label{BSR-opt-smoothing}
The optimal smoothing factor for both the Poisson-stabilized and projection stabilized discretizations with exact BSR is
  \begin{equation*}
   \mu_{{\rm opt}}=\displaystyle \min_{(\alpha, \omega)}\max_{ \boldsymbol{\theta}\in T^{{\rm high}}}{\big|\lambda(\mathcal{\widetilde{S}}(\alpha, \omega,\boldsymbol{\theta}))\big|}=\frac{1}{3},
\end{equation*}
if and only if
\begin{equation*}
\frac{\omega}{\alpha}=\frac{8}{9},\,\,\frac{3}{4}\leq \alpha \leq \frac{3}{2}.
\end{equation*}

\end{theorem}
\begin{proof}
 Note that $a\in[2,4]$, and choose $\alpha$ such that $2=a_{\rm{min}} \leq \frac{8}{3}\alpha\leq a_{\rm{max}}=4$. If $c$ is positive, the following always holds
\begin{equation*}
\frac{3}{4\alpha}=\frac{a_{\rm{min}}}{\frac{8}{3}\alpha}\leq \frac{a_{\rm{min}}c+b}{\frac{8}{3}\alpha c+b}\leq \frac{ac+b}{\frac{8}{3}\alpha c+b}\leq \frac{a_{\rm{max}}c+b}{\frac{8}{3}\alpha c+b}\leq \frac{a_{\rm{max}}}{\frac{8}{3}\alpha}=\frac{3}{2\alpha}.
\end{equation*}
Furthermore, if $\frac{\omega}{\alpha}=\frac{8}{9}$, we have
\begin{equation}\label{lambda3-bound}
 \frac{2}{3}=\frac{3}{4\alpha}\cdot\frac{8}{9}\alpha\leq \omega \lambda_3 \leq\frac{3}{2\alpha}\cdot\frac{8}{9}\alpha=\frac{4}{3}.
\end{equation}
For both discretizations, we can check that $c>0$ over the high frequencies.   From (\ref{lambda3-bound}), it is easy to see that $|1-\omega\lambda_3|\leq \frac{1}{3}$, with $\alpha=\frac{9}{8}\omega\in[\frac{3}{4},\frac{3}{2}]$.
\end{proof}

\subsection{Inexact Braess-Sarazin relaxation}
Here, we also consider solving the Schur complement equation,  (\ref{schur-solution-of-precondtion}), by weighted Jacobi relaxation with weight, $\omega_J$. Following \cite{MR1810326}, we refer to this as  inexact Braess-Sarazin relaxation (IBSR). Let the corresponding block preconditioner be $M_{I}$, given by

\begin{equation*}
   M_{I}=  \begin{pmatrix}
      \alpha D & B^{T}\\
     B & \hat{S}+B(\alpha D)^{-1}B^{T}\\
    \end{pmatrix}
\end{equation*}
where $\hat{S}$ is the approximation of  $-S=-B(\alpha D)^{-1}B^{T}-C$ used in (\ref{schur-solution-of-precondtion}). For one sweep of weighted Jacobi relaxation, $\hat{S}$ is given by
\begin{equation*}
  \hat{S}_1 = -\frac{1}{\omega_J}{\rm diag}(S),
\end{equation*}
and for 2 sweeps of weighted Jacobi relaxation with equal weights, $\hat{S}$ is given by
\begin{equation*}
    \hat{S}_2 = \hat{S}_1\Big(2I+\hat{S}_1^{-1}S\Big)^{-1}.
\end{equation*}
By direct computation, $B(\alpha D)^{-1}B^T:=S_{0}$ can be written in terms of a $5\times 5$ stencil:
\begin{equation}\label{5-stencil-Jacobi}
S_{0}=\frac{h^2}{\alpha}\begin{pmatrix}
      -1/192  &        -1/48    &      -1/24    &      -1/48    &      -1/192\\
      -1/48   &        0    &          1/24    &       0      &       -1/48\\
      -1/24    &       1/24  &         3/16     &      1/24    &      -1/24\\
      -1/48    &       0      &        1/24      &     0        &     -1/48\\
      -1/192   &      -1/48    &      -1/24       &   -1/48      &    -1/192\\
\end{pmatrix}.
\end{equation}
The symbol of $S_{0}$ is $\widetilde{S_0}=\frac{3b}{8\alpha}:=\varsigma$ for $b=-b_1^2-b_2^2$. In fact, $\varsigma =\widetilde{B}(\widetilde{\alpha D})^{-1}\widetilde{B^T}$. Let $\gamma$ be the symbol of $\hat{S}_1$,
\begin{equation*}
\gamma = \left\{\begin{array}{cl} -\frac{h^2}{24\omega_J}(\frac{9}{2\alpha}+\frac{8}{3}), & \mbox{ for PoSD}\\
    -\frac{h^2}{24\omega_J}(\frac{9}{2\alpha}+\frac{14}{3}), & \mbox{ for PrSD}\\
    \end{array}\right.
\end{equation*}
Similarly, let $\eta$ be the symbol of $\hat{S}+B(\alpha D)^{-1}B^{T}$,
\begin{equation*}
\eta = \left\{\begin{array}{cl}  \gamma+\varsigma, & \mbox{for  one sweep} \,\, (\hat{S}_1)\\
    \Big(2+\tau\gamma^{-1}\Big)^{-1}\gamma+\varsigma, & \mbox{for two sweeps} \,\, (\hat{S}_2)\\
    \end{array}\right.
\end{equation*}
where
\begin{equation*}
\tau = \left\{\begin{array}{cl} \varsigma+c_1, & \mbox{ for PoSD}\\
    \varsigma+c_2, & \mbox{ for PrSD}\\
    \end{array}\right.
\end{equation*}
Finally, the symbol of $M_{I}$ is given by
\begin{equation}\label{Stable-IBSR-sym-MI}
   \widetilde{M}_{I}(\theta_1,\theta_2) =\begin{pmatrix}
       \frac{8}{3}\alpha & 0 & b_1  \\
      0 &   \frac{8}{3}\alpha & b_2 \\
      -b_1   &  -b_2 &\eta
    \end{pmatrix}.
\end{equation}
The symbol of the error-propagation matrix for IBSR is $\mathcal{\widetilde{S}}_{I}(\alpha, \omega,\boldsymbol{\theta})=I- \omega \widetilde {M}_{I}^{-1}\widetilde{\mathcal{L}}$.
A standard calculation shows that the determinant of $\widetilde{\mathcal{L}}-\lambda \widetilde{M}_{I}$ is
\begin{eqnarray}\label{eigs-IBSR}
 \pi_{I}(\lambda;\alpha,\omega,\omega_J) = -(a-\frac{8}{3}\alpha \lambda)\bigg[(b-\frac{8\alpha\eta}{3})\lambda^2+(a\eta-\frac{8\alpha c}{3}-2b)\lambda+ac+b\bigg].
\end{eqnarray}

From (\ref{eigs-IBSR}), we see there is an eigenvalue $\frac{3a}{8\alpha}$, which is the same as that of exact BSR. As before, the question now becomes whether there is a choice of $\omega,\alpha$  and $\omega_J$ such that convergence equal to that of exact BSR can be achieved. We leave this as an open question for future work and, instead, numerically optimize the two-grid convergence factor over these parameters.

\begin{remark}
A similar form to (\ref {Stable-IBSR-sym-MI}) occurs for inexact BSR applied to the stable $Q_2-Q_1$ approximation, modifying the stencil of $C$ to be zero, and accounting for the block structure shown in (\ref{Full-symbol}).
\end{remark}

 \subsection{Numerical experiments for stabilized discretizations}\label{sec:Numer}
We now present LFA predictions, validating DWJ, (I)BSR, and the related Uzawa iteration  against
 measured multigrid performance for these schemes. We consider the homogeneous problem in (\ref{Stokes1-FEM}),  with periodic boundary conditions, and a random initial guess, $x_h^{(0)}$.

Convergence is measured using the averaged convergence  factor, $\hat{\rho}_{h}=\sqrt[k]{\frac{\|d_{h}^{(k)}\|_{2}}{\|d_{h}^{(0)}\|_{2}}}$, with $k=100$, and $d_{h}^{(k)}=b-Kx_h^{(k)}$. The LFA predictions are made with $h=1/128$, for both the smoothing factor, $\mu$, and two-grid convergence factor, $\rho_h$. For testing, we use standard $W(\nu_1,\nu_2)$ cycles with bilinear interpolation for $Q_1$ variables and biquadratic interpolation for $Q_2$ variables, and their adjoints for restriction. We consider both rediscretization and Galerkin coarsening, noting that they coincide for all terms except the stabilization terms that include a scaling of $h^2$. The coarsest grid is a mesh with 4 elements. Where significant differences arise, we also report two-grid convergence rates for $TG(\nu_1,\nu_2)$ cycles.

\subsubsection{PoSD with DWJ}

From the range of parameters allowed in (\ref{beta-DWJ-Parameter-domain}), we select $\alpha_1=1.451,\,\,\alpha_2=1.000$, and $\omega=1.290$ (for convenience, satisfying the equality in (\ref{beta-DWJ-Parameter-domain})) to compute the LFA predictions. Figure \ref{beta-dis} shows the spectrum of the two-grid error-propagation operators for DWJ relaxation with rediscretization and Galerkin coarsening. Note that  the two-grid  convergence factor is the same as the optimal smoothing factor for rediscretization, but not for Galerkin coarsening.

\begin{figure}[H]
\centering
\includegraphics[width=6.5cm,height=5.5cm]{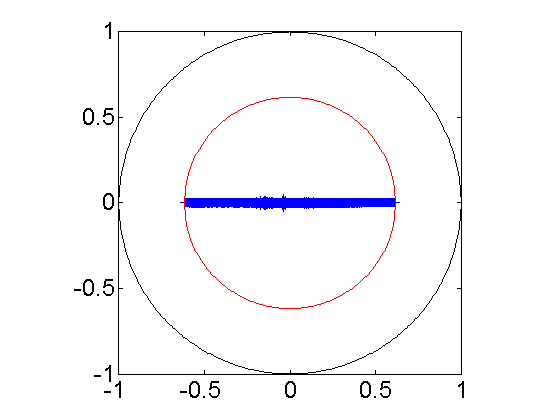}
\includegraphics[width=6.5cm,height=5.5cm]{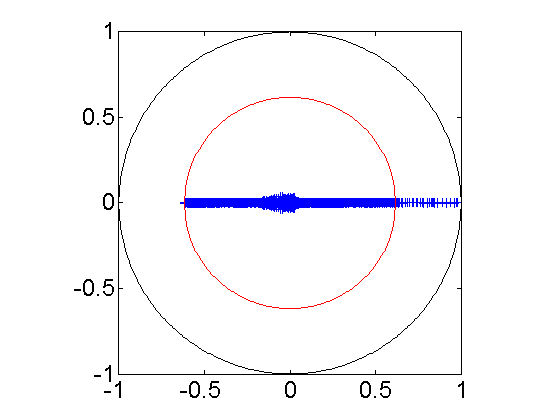}
\caption{The spectrum of the two-grid error-propagation operator using DWJ for PoSD. Results with rediscretization are shown at left, while those with Galerkin coarsening are at right. In both figures, the inner circle has radius equal to the LFA smoothing factor.} \label{beta-dis}
\end{figure}

In order to see the sensitivity of performance to parameter choice, we consider the two-grid LFA convergence factor  with rediscretization coarsening. From (\ref{beta-DWJ-Parameter-domain}), we know that there are many optimal parameters.  To fix a single parameter for DWJ, we consider the case of $\omega=\frac{459}{356}$ and, at the left of Figure \ref{beta-omega-alpha1}, we present the LFA-predicted two-grid convergence factors for DWJ with variation in $\alpha_1$ and $\alpha_2$. Here, we see strong sensitivity to ``too small'' values of both parameters, for $\alpha_1<1$ and $\alpha_2<0.9$, including a notable portion of the optimal range of values predicted by the LFA smoothing factor.  At the right of Figure \ref{beta-omega-alpha1},  we fix $\alpha_2=\frac{356}{459}\omega$ and vary $\omega$ and $\alpha_1$.  The two lines are the lower and upper bounds from (\ref{beta-DWJ-Parameter-domain}), between which LFA predicts the optimal convergence factor should be achieved. Note that not all of the allowed parameters obtain the optimal convergence factor.   Here, we see great sensitivity for large values of $\omega$, but a large range with  generally similar performance as in the optimal parameter case.

 \begin{figure}[H]
\centering
\includegraphics[width=6.5cm,height=5.5cm]{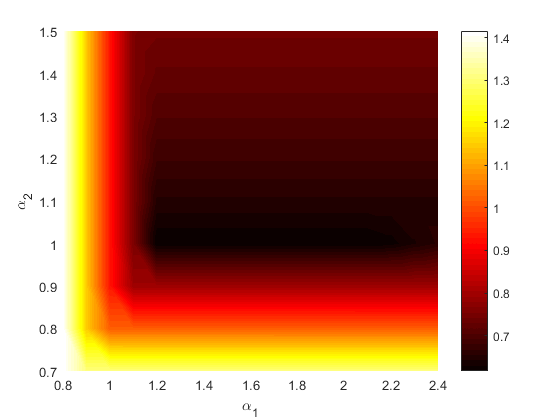}
\includegraphics[width=6.5cm,height=5.5cm]{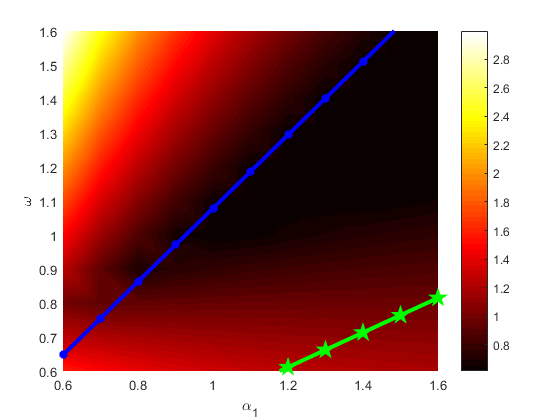}
\caption{The two-grid LFA convergence factor for the PoSD using DWJ  and rediscretization. At left, we fix $\omega=\frac{459}{356}$ and vary $\alpha_1$ and $\alpha_2$. At right, we fix $\alpha_2=\frac{356}{459}\omega$ and vary $\omega$ and $\alpha_1$.} \label{beta-omega-alpha1}
\end{figure}

In Table \ref{stabilized-R-W}, we present the multigrid performance of DWJ with $W$-cycles for rediscretization coarsening.
These results show measured multigrid convergence factors that coincide with the LFA-predicted two-grid convergence factors.
Similar results are seen for $V$-cycles with rediscretization. For Galerkin coarsening, nearly identical $W$-cycle results are seen when $\nu_1+\nu_2>2$, but divergence is seen for $W$-cycles with $\nu_1+\nu_2=1$ or 2, and for all $V$-cycles tested. In Table \ref{stabilized-R-W-2Jacobi}, we report the multigrid performance of DWJ using 2 sweeps of Jacobi relaxation on the pressure equation with rediscretization for PoSD. Here, we take $\alpha_1=3/2, \omega_J=1, \omega=4/3$ as in Theorem \ref{optimal-PoSD-DWJ-2Jacobi}. We see that the LFA convergence factor accurately  predicts the measured  performance.

\begin{table}[H]
 \caption{$W$-cycle convergence factors, $\hat{\rho}_h$, for DWJ with rediscretization for PoSD, compared with LFA two-grid predictions, $\rho_h$. Here, the algorithmic parameters are $\alpha_1=1.451, \alpha_2=1.000, \omega=1.290$ and the LFA smoothing factor is $\mu=0.618$.}
\centering
\begin{tabular}{|l|c|c|c|c|c|c|}
\hline
Cycle  & $W(0,1)$  & $W(1,0)$   &$W(1,1)$   & $W(1,2)$  &$W(2,1)$   &$W(2,2)$  \\
\hline
\hline
$\rho_{h=1/128}$          &0.618     &0.618    &0.382     &0.236    &0.236    &0.146\\
\hline
$\hat{\rho}_{h=1/64}$     &0.564     & 0.568    &0.349      & 0.215      &0.214     &0.133 \\
\hline
$\hat{\rho}_{h=1/128}$     &0.561     & 0.568    &0.348      & 0.215      &0.214     &0.132 \\
\hline
\end{tabular}\label{stabilized-R-W}
\end{table}

\begin{table}[H]
 \caption{$W$-cycle convergence factors, $\hat{\rho}_h$, for DWJ with 2 sweeps of Jacobi on the pressure equation for PoSD with rediscretization, compared with LFA two-grid predictions, $\rho_h$. Here, the algorithmic parameters are $\alpha_1=3/2, \omega_J=1, \omega=4/3$ and the LFA smoothing factor is $\mu=0.333$.}
\centering
\begin{tabular}{|l|c|c|c|c|c|c|}
\hline
Cycle  & $W(0,1)$  & $W(1,0)$   &$W(1,1)$   & $W(1,2)$  &$W(2,1)$   &$W(2,2)$  \\
\hline
\hline
$\rho_{h=1/128}$          &0.338     &0.338    &0.115     &0.078    &0.078    &0.061\\
\hline
$\hat{\rho}_{h=1/64}$     &0.324     &0.324    &0.112     &0.074    &0.075    &0.074\\
\hline
$\hat{\rho}_{h=1/128}$     &0.324     &0.324    &0.112     &0.075    &0.075    &0.073\\
\hline
\end{tabular}\label{stabilized-R-W-2Jacobi}
\end{table}

\subsubsection{PrSD with DWJ}

From the range of parameters allowed in (\ref{D-DWJ-Parameter-domain}), we choose $\alpha_1=1,\,\,\alpha_2=1,\,\,\omega=\frac{108}{97}$. Figure \ref{D-dis-eig} shows that the smoothing factor provides a good prediction for the two-grid convergence factor with rediscretization, but not with Galerkin coarsening.
\begin{figure}[H]
\centering
\includegraphics[width=6.5cm,height=5.5cm]{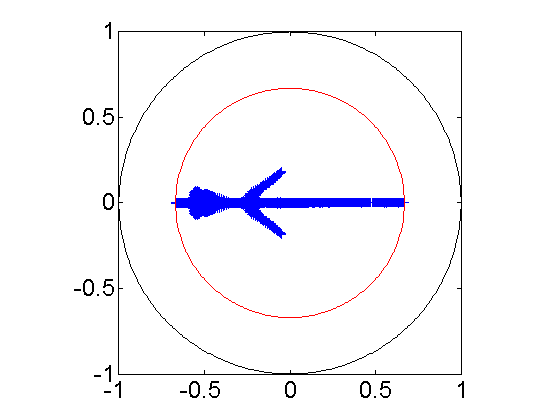}
\includegraphics[width=6.5cm,height=5.5cm]{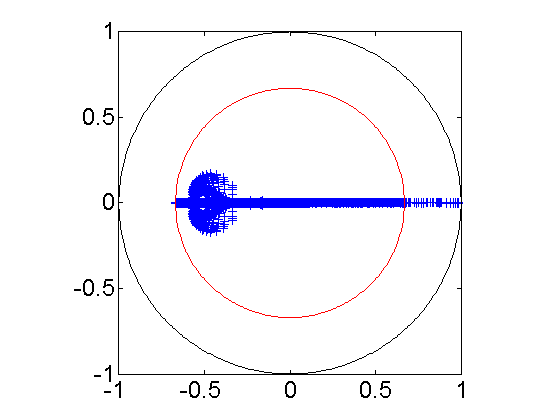}
\caption{The spectrum of the two-grid error-propagation operator using DWJ for PrSD. Results with rediscretization are shown at left, while those with Galerkin coarsening are at right. In both figures, the inner circle has radius equal to the LFA smoothing factor.} \label{D-dis-eig}
\end{figure}


Similarly to the discussion above, we consider the sensitivity to parameter choice for DWJ applied to PrSD. To fix a single parameter for DWJ, we consider  the case of $\omega=\frac{108}{97}$.   At the left of Figure \ref{D-omega-alpha1}, we present the LFA-predicted convergence factors for DWJ with variation in $\alpha_1$ and $\alpha_2$, again seeing a strong sensitivity to ``too small'' values of the parameters. At the right of Figure \ref{D-omega-alpha1}, we fix $\alpha_2=\frac{97}{108}\omega$.  The two lines are the lower and upper bounds from  (\ref{D-DWJ-Parameter-domain}), between which LFA predicts the optimal convergence factor should be achieved. Note that not all of the parameters in this range obtain the optimal convergence factor. We see that, for small $\alpha_1$, the convergence factor is very sensitive to large values of $\omega$.
 \begin{figure}[H]
\centering
\includegraphics[width=6.5cm,height=5.5cm]{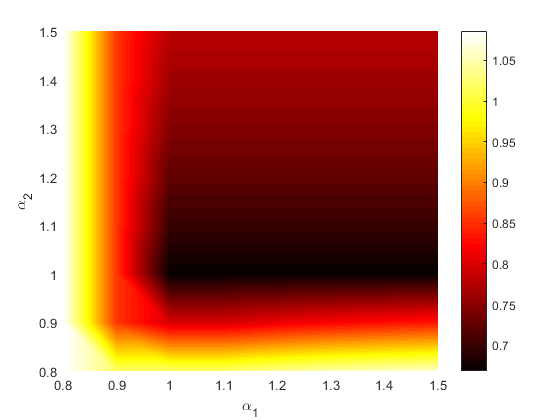}
\includegraphics[width=6.5cm,height=5.5cm]{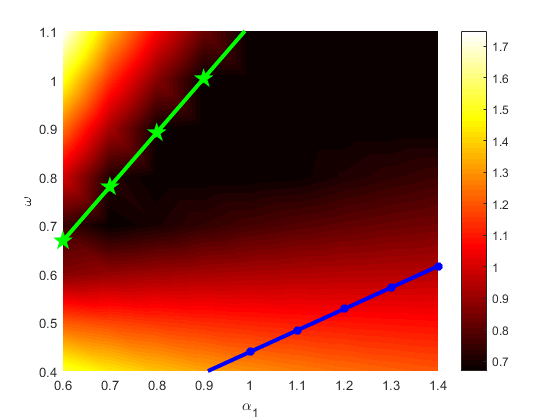}
\caption{The two-grid LFA convergence factor for the PrSD using DWJ  and rediscretization. At left, we fix $\omega=\frac{108}{97}$ and vary $\alpha_1$ and $\alpha_2$. At right, we fix $\alpha_2=\frac{97}{108}\omega$ and vary $\omega$ and $\alpha_1$.} \label{D-omega-alpha1}
\end{figure}


In Table \ref{W-Dohrmann}, we present the multigrid performance of DWJ relaxation with $W$-cycles for rediscretization coarsening.
We see that the measured multigrid convergence factors match well with the LFA-predicted two-grid convergence factors.  For Galerkin
 coarsening, as in the case of PoSD, we see divergence when $\nu_1+\nu_2\leq2$, but performance matching that of rediscretization for $\nu_1+\nu_2>2$. Here, $V$-cycle results are similar to the $W$-cycle results for both rediscretization and Galerkin coarsening approaches. In Table \ref{PrSD-W-2Jacobi}, we compare the LFA predictions with multigrid performance for DWJ using 2 sweeps of Jacobi relaxation on the pressure equation. Here, we take $\alpha_1=3/2, \omega_J=1, \omega=4/3$ as in Theorem \ref{optimal-PrSD-DWJ-2Jacobi}, and observe a good match between the LFA predictions and measured performance.

\begin{table}[H]
 \caption{$W$-cycle convergence factors, $\hat{\rho}_h$, for DWJ with rediscretization for PrSD, compared with LFA two-grid predictions, $\rho_h$. Here, the algorithmic parameters are $\alpha_1=1, \alpha_2=1, \omega=108/97$ and  the LFA smoothing factor is $\mu=0.670$.}
\centering
\begin{tabular}{|l|c|c|c|c|c|c|}
\hline
Cycle   & $W(0,1)$  & $W(1,0)$   &$W(1,1)$   & $W(1,2)$  &$W(2,1)$   &$W(2,2)$  \\
\hline
\hline
$\rho_{h=1/128}$          &0.670     &0.670      &0.449     &0.300    &0.300     &0.201\\
\hline
$\hat{\rho}_{h=1/64}$     & 0.652    &0.652       &0.436      &0.291      &0.292    &0.196   \\
\hline

$\hat{\rho}_{h=1/128}$     &0.651      &0.652     &0.435     &0.291     & 0.291    & 0.195 \\
\hline
\end{tabular}\label{W-Dohrmann}
\end{table}

\begin{table}[H]
 \caption{$W$-cycle convergence factors, $\hat{\rho}_h$, for DWJ with 2 sweeps of Jacobi on the pressure equation for PrSD with rediscretization, compared with LFA two-grid predictions, $\rho_h$. Here, the algorithmic parameters are $\alpha_1=3/2, \omega_J=1, \omega=4/3$ and the LFA smoothing factor is $\mu=0.333$.}
\centering
\begin{tabular}{|l|c|c|c|c|c|c|}
\hline
Cycle  & $W(0,1)$  & $W(1,0)$   &$W(1,1)$   & $W(1,2)$  &$W(2,1)$   &$W(2,2)$  \\
\hline
\hline
$\rho_{h=1/128}$          &0.333     &0.333    &0.112     &0.079    &0.079    &0.062\\
\hline
$\hat{\rho}_{h=1/64}$     &0.324     &0.324    &0.112     &0.074    &0.075    &0.074\\
\hline
$\hat{\rho}_{h=1/128}$     &0.324     &0.324    &0.112     &0.075    &0.075    &0.073\\
\hline
\end{tabular}\label{PrSD-W-2Jacobi}
\end{table}

\subsubsection{PoSD with BSR}
Next, we consider BSR for PoSD, first displaying the two-grid LFA convergence factor as a function of $\alpha$ for rediscretization coarsening with $\omega=\frac{8}{9}\alpha$ in Figure \ref{beta-BSR-CS(1,1)}. Comparing the convergence factor with $\mu^{2}$, for $\nu_1=\nu_2=1$, we see a good match over the interior of the interval $\frac{3}{4}\leq \alpha\leq \frac{3}{2}$ predicted by Theorem \ref{BSR-opt-smoothing}. For larger values of $\nu_1+\nu_2$, this agreement deteriorates as is typical when the behavior of coarse-grid correction becomes dominant.
At the right of Figure \ref{beta-BSR-CS(1,1)}, we see good agreement between $\rho$ and $\mu$ when $\nu_1+\nu_2=1$ with fixed $\alpha=1$. In both cases, similar behaviour is seen with Galerkin coarsening.

\begin{figure}[H]
\centering
\includegraphics[width=6.5cm,height=5.5cm]{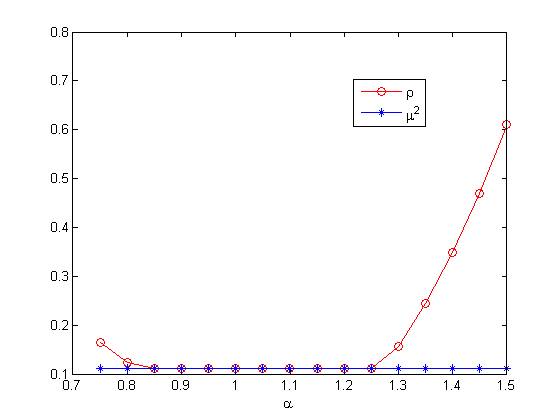}
\includegraphics[width=6.5cm,height=5.5cm]{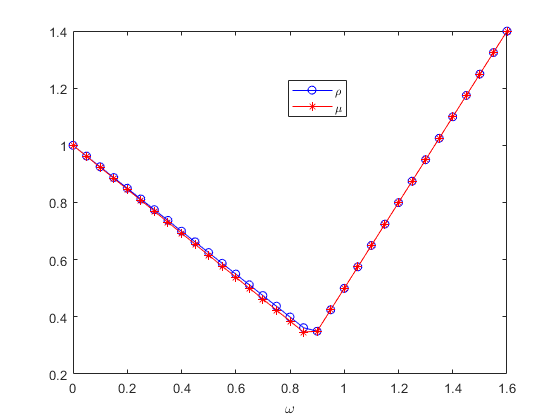}
\caption{Two-grid and smoothing factors for BSR with rediscretization for PoSD. At left, comparing $\rho$ with $\mu^2$ for $\nu_1=\nu_2=1$ with $\omega=\frac{8}{9}\alpha$. At right, comparing $\rho$ with $\mu$ for $\nu_1+\nu_2=1$ with $\alpha=1$.}  \label{beta-BSR-CS(1,1)}
\end{figure}

Motivated by the above, we use $\alpha =1$ and $\omega=\frac{8}{9}$ for multigrid experiments with rediscretization, solving the Schur complement equation exactly.
 Table \ref{exact-BSR-stabilized-R-W} shows that the measured multigrid convergence factors match well with the LFA-predicted two-grid
 convergence factors for $W$-cycles with rediscretization coarsening, and similar results are seen for Galerkin coarsening.

\begin{table}[H]
 \caption{$W$-cycle convergence factors, $\hat{\rho}_h$, for BSR with rediscretization for PoSD, compared with LFA two-grid predictions, $\rho_h$. Here, the algorithmic parameters are $\alpha=1, \omega=\frac{8}{9}$ and the LFA smoothing factor is $\mu=0.333$.}
\centering
\begin{tabular}{|l|c|c|c|c|c|c|}
\hline
Cycle  & $W(0,1)$  & $W(1,0)$   &$W(1,1)$   & $W(1,2)$  &$W(2,1)$   &$W(2,2)$  \\
\hline
\hline
$\rho_{h=1/128}$                   &0.333    &0.333    &0.111     &0.079    & 0.079    &0.062  \\
\hline
$\hat{\rho}_{h=1/64}$     & 0.324    &0.323    &0.112     &0.075    &0.075     &0.058 \\
\hline
$\hat{\rho}_{h=1/128}$     &0.323    &0.323         &0.112      &0.075       &0.075     &0.058 \\
\hline
\end{tabular}\label{exact-BSR-stabilized-R-W}
\end{table}

In Table \ref{PoSD-LFA-OPT-Prediction}, we report the LFA prediction for IBSR with different parameters and one or two sweeps of Jacobi relaxation on the approximate Schur complement for PoSD. Here, we clearly see the benefit of two sweeps of relaxation on the approximate Schur complement over a single sweep, as well as that better performance is possible when (numerically) optimizing all of the parameters for IBSR independent of the optimization for exact BSR.

\begin{table}[H]
\caption{LFA predictions: two-grid convergence factor, $\rho_{h=1/128}$, and smoothing factor, $\mu$, of IBSR with rediscretization for PoSD with $\nu_1+\nu_2=1$.}
\centering
\begin{tabular}{|l|c|c|c|c|c|}
\hline
Sweep  & $\alpha$  & $\omega$   &$\omega_J$   &$\mu$  &$\rho_{h=1/128}$  \\
\hline
\hline
1 (Optimized)                 &1.2    &1.1    &0.7     &0.679   &0.679     \\
\hline
1                           &1.0    &8/9    & 1.0    &0.669   &0.735     \\
\hline
2 (Optimized)                &1.1    &1.0    &1.0     &0.366   &0.366     \\
\hline
2                  &1.0    &8/9    &1.0     &0.461   &0.461     \\
\hline
\end{tabular}\label{PoSD-LFA-OPT-Prediction}
\end{table}

Table \ref{IBSR-PoSD-TG-Jacobi-opt} shows that the LFA-predicted 2-grid convergence factors closely match those seen in practice. However, as shown  in Table \ref{IBSR-PoSD-2-Jaocbi}, significant degradation is seen when considering W-cycles, particularly as $\nu_1+\nu_2$ increases.

\begin{table}[H]
 \caption{Two-grid convergence factor, $\hat{\rho}_h$  for IBSR with 2 sweeps of Jacobi with rediscretization for PoSD, compared with LFA two-grid predictions, $\rho_h$, with optimized parameters, $\alpha=1.1,\omega=1.0$ and $\omega_J=1$.}
\centering
\begin{tabular}{|l|c|c|c|c|c|c|}
\hline
Cycle   & $TG(0,1)$  & $TG(1,0)$   &$TG(1,1)$   & $TG(1,2)$  &$TG(2,1)$   &$TG(2,2)$  \\
\hline
\hline
$\rho_{h=1/128}$                   &0.366     &0.366   &0.167     &0.128   & 0.128    &0.106  \\
\hline
$\hat{\rho}_{h=1/64}$      &0.352     &0.353   &0.160     &0.120    & 0.120    &0.100  \\
\hline
$\hat{\rho}_{h=1/128}$     &0.352     &0.353   &0.160     &0.122    & 0.122    &0.100  \\
\hline
\end{tabular}\label{IBSR-PoSD-TG-Jacobi-opt}
\end{table}

\begin{table}[H]
 \caption{$W$-cycles convergence factors, $\hat{\rho}_h$,  for IBSR with 2 sweeps of Jacobi with rediscretization for PoSD, compared with LFA two-grid predictions, $\rho_h$  with optimized parameters, $\alpha=1.1, \omega=1.0$ and $\omega_J=1$.}
\centering
\begin{tabular}{|l|c|c|c|c|c|c|}
\hline
Cycle   & $W(0,1)$  & $W(1,0)$   &$W(1,1)$   & $W(1,2)$  &$W(2,1)$   &$W(2,2)$  \\
\hline
\hline
$\rho_{h=1/128}$                 &0.366     &0.366   &0.167     &0.128   & 0.128    &0.106  \\
\hline
$\hat{\rho}_{h=1/64}$        &0.456     &0.453  &0.245     &0.197    &0.200    &0.167 \\
\hline
$\hat{\rho}_{h=1/128}$       &0.459     &0.462    &0.257     &0.206    &0.211    &0.175 \\
\hline
\end{tabular}\label{IBSR-PoSD-2-Jaocbi}
\end{table}

The gap between the results seen for exact BSR in Table \ref{exact-BSR-stabilized-R-W} and those for IBSR in Table \ref{IBSR-PoSD-2-Jaocbi} is quite significant. To maintain the performance observed for exact BSR, we could simply use  more Jacobi iterations on the Schur complement system in IBSR; however, experiments showed that this did not lead to a scalable algorithm. Instead, we consider solving the Schur complement system by applying a multigrid $W(1,1)$-cycle using weighted Jacobi relaxation with weight $\omega_J$, shown in Table \ref{inner-W-PoSD-Redis}. From Table \ref{inner-W-PoSD-Redis}, we observe that using only 1 or 2 $W(1,1)$-cycles on the approximate Schur complement achieves  convergence factors essentially  matching those in Table \ref{exact-BSR-stabilized-R-W}, showing that the $W(1,1)$ cycle is the most cost effective.


\begin{table}[H]
 \caption{$W$-cycle convergence factors, $\hat{\rho}_h$, for IBSR  with inner $W(1,1)$-cycle for PoSD and $(\alpha,\omega,\omega_J)=(1,8/9,1)$. In brackets, minimum value of the number of inner $W(1,1)$-cycles that achieves the same convergence factors as those of LFA predictions, $\rho_h$, for exact BSR.}
\centering
\begin{tabular}{|l|l|l|l|l|l|l|}
\hline
Cycle   & $W(0,1)$  & $W(1,0)$   &$W(1,1)$   & $W(1,2)$  &$W(2,1)$   &$W(2,2)$  \\
\hline
\hline
$\rho_{h=1/128}$       &0.333    &0.333    &0.111     &0.079    & 0.079    &0.062  \\
\hline
$\hat{\rho}_{h=1/64}$     &0.368(2)    &0.346(2)    &0.131(2)     &0.075(2)    &0.075(2)   &0.059(1)  \\
\hline
$\hat{\rho}_{h=1/128}$     &0.343(2)    &0.351(2)   &0.111(2)     &0.075(2)    &0.075(2)     &0.063(1)  \\
\hline
\end{tabular}\label{inner-W-PoSD-Redis}
\end{table}

\subsubsection{PrSD with BSR}
We now consider BSR for  PrSD. At the left of Figure \ref{D-BSR-CS(1,0)}, we see a good agreement between the two-grid convergence factor and $\mu^{2}$ for $\nu_1=\nu_2=1$ for some parameters in the range defined in Theorem \ref{BSR-opt-smoothing} when using rediscretization. A larger interval of agreement is seen for the corresponding results for Galerkin coarsening.  In both cases, agreement between the two-grid convergence factor and $\mu^{\nu_1+\nu_2}$ degrades as $\nu_1+\nu_2$ increases, as expected.

 Note that Theorem \ref{BSR-opt-smoothing} demonstrates that the smoothing factor for BSR is a function of
   $\frac{\omega}{\alpha}$ (but the same is not necessarily true for the convergence factor). In Figure \ref{D-BSR-CS(1,0)}, we plot the LFA smoothing and convergence factors for BSR with rediscretization as a function of $\omega$, with $\alpha=0.8$ and see that these factors generally agree, although the smoothing factor slightly underestimates the convergence factor. As two-grid convergence  is, however, sensitive to the choice of $\alpha$, the smoothing factor generally underestimates the convergence factor for other values of $\alpha$.

\begin{figure}[H]
\centering
\includegraphics[width=6.5cm,height=5.5cm]{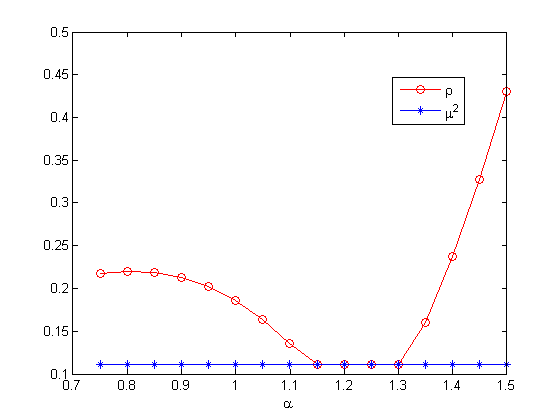}
\includegraphics[width=6.5cm,height=5.5cm]{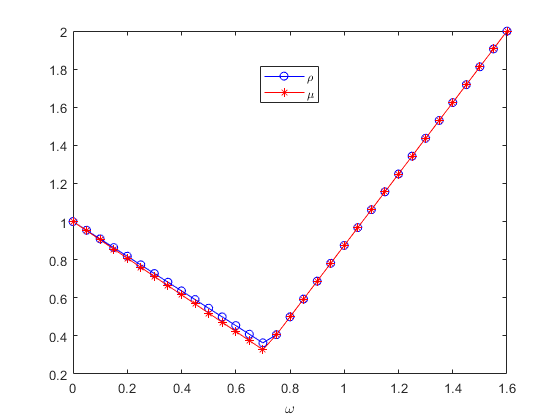}
\caption{Two-grid and smoothing factors for BSR with rediscretization for PrSD. At left, comparing $\rho$ with $\mu^2$ for $\nu_1=\nu_2=1$ with $\omega=\frac{8}{9}\alpha$. At right, comparing $\rho$ with $\mu$ for $\nu_1+\nu_2=1$ with $\alpha=\frac{4}{5}$.} \label{D-BSR-CS(1,0)}
\end{figure}

 Fixing $\omega=\frac{8}{9}\alpha$ with $\alpha =1.2$ (as suggested by Figure \ref{D-BSR-CS(1,0)} for $\nu_1=\nu_2=1$), Table \ref{exact-BSR-W-Dohrmann} shows that the measured multigrid convergence factors again match well with the LFA-predicted two-grid convergence factors for  $W$-cycles with  rediscretization coarsening. Note, however, the degradation for $\nu_1+\nu_2=1$, where the smoothing factor analysis predicts a convergence factor of $\frac{1}{3}$ that is not realized. The convergence factor of $\frac{1}{3}$ can be achieved by choosing $\alpha=\frac{4}{5}$ and $\omega=\frac{8}{9}\alpha$ in the BSR scheme with either $W(1,0)$ or $W(0,1)$ cycles, but these choices lead to  a slight degradation  with $\nu_1+\nu_2>1$.  Similar results are seen for Galerkin coarsening with $\alpha=1$ and $\omega=\frac{8}{9}\alpha$ with the notable exception that the smoothing factor prediction was matched by both the two-grid LFA convergence factor and true $W$-cycle convergence in this case for all experiments.

\begin{table}[H]
 \caption{$W$-cycle convergence factors, $\hat{\rho}_h$, for BSR with rediscretization for PrSD, compared with LFA two-grid predictions, $\rho_h$, with algorithmic parameters, $\alpha=1.2$ and $\omega=\frac{8}{9}\alpha$.}
\centering
\begin{tabular}{|l|c|c|c|c|c|c|}
\hline
Cycle   & $W(0,1)$  & $W(1,0)$   &$W(1,1)$   & $W(1,2)$  &$W(2,1)$   &$W(2,2)$  \\
\hline
\hline
$\rho_{h=1/128}$                   &0.673     & 0.673     &0.111     &0.079    &0.079     &0.062 \\
\hline
$\hat{\rho}_{h=1/64}$      &0.585     &0.585      &0.112     &0.075    &0.075     &0.058  \\
\hline

$\hat{\rho}_{h=1/128}$     &0.584     &0.584      &0.112     &0.075    &0.075     &0.058            \\
\hline
\end{tabular}\label{exact-BSR-W-Dohrmann}
\end{table}

In Table \ref{PrSD-LFA-OPT-Prediction}, we report the LFA prediction for IBSR with one or two sweeps of Jacobi relaxation on the approximate Schur complement with different parameters for PrSD. As in the case of PoSD, one sweep is not enough to obtain performance comparable to exact BSR, and there is a significant advantage to independently optimizing the parameters for IBSR.

\begin{table}[H]
\caption{LFA predictions: two-grid convergence factor, $\rho_{h=1/128}$, and smoothing factor, $\mu$, of IBSR with rediscretization for PrSD with $\nu_1+\nu_2=1$.}
\centering
\begin{tabular}{|l|c|c|c|c|c|}
\hline
Sweep  & $\alpha$  & $\omega$   &$\omega_J$   &$\mu$  &$\rho_{h=1/128}$  \\
\hline
\hline
1 (Optimized)                 &1.6    &0.8   &1     &0.714   &0.714     \\
\hline
1                   &1.2    &16/15    &1.0    & 0.718  &1.027    \\
\hline
2 (Optimized)                &1.2    &0.9    &1.2    &0.494   &0.445     \\
\hline
2                   &1.2    &16/15    &1.0    &0.431   &0.549     \\
\hline
\end{tabular}\label{PrSD-LFA-OPT-Prediction}
\end{table}

Considering, then, the two-grid method with these optimized parameters and two relaxation sweeps on the approximate Schur complement, Table \ref{IBSR-TG-PrSD-2-sweep-opt} shows that two-grid LFA offers a good prediction of performance. In Table \ref{IBSR-PrSD-2-sweeps}, however, we see degraded performance when using $W$-cycles.

\begin{table}[H]
 \caption{Two-grid convergence factors, $\hat{\rho}_h$  for IBSR with 2 sweeps of Jacobi for PrSD  with rediscretization, compared with LFA two-grid predictions, $\rho_h$, with optimized parameters, $\alpha=1.2,\omega=0.9$ and $\omega_J=1.2$.}
\centering
\begin{tabular}{|l|c|c|c|c|c|c|}
\hline
Cycle   & $TG(0,1)$  & $TG(1,0)$   &$TG(1,1)$   & $TG(1,2)$  &$TG(2,1)$   &$TG(2,2)$  \\
\hline
\hline
$\rho_{h=1/128}$                 &0.445    & 0.445     &0.319     &0.262   &0.262     &0.225 \\
\hline
$\hat{\rho}_{h=1/64}$     &0.418    & 0.420     &0.301     &0.251    &0.250     &0.212  \\
\hline

$\hat{\rho}_{h=1/128}$     &0.420   &0.420      &0.304     &0.250     &0.249    &0.212  \\
\hline
\end{tabular}\label{IBSR-TG-PrSD-2-sweep-opt}
\end{table}

\begin{table}[H]
\caption{$W$-cycles convergence factors, $\hat{\rho}_h$  for IBSR with 2 sweeps of Jacobi  for PrSD  with rediscretization, compared with LFA two-grid predictions, $\rho_h$, with optimized parameters, $\alpha=1.2,\omega=0.9$ and $\omega_J=1.2$.}
\centering
\begin{tabular}{|l|c|c|c|c|c|c|}
\hline
Cycle   & $W(0,1)$  & $W(1,0)$   &$W(1,1)$   & $W(1,2)$  &$W(2,1)$   &$W(2,2)$  \\
\hline
\hline
$\rho_{h=1/128}$                  &0.445    & 0.445     &0.319     &0.262   &0.262     &0.225 \\
\hline
$\hat{\rho}_{h=1/64}$     &0.739     &0.740     &0.340   &0.304   &0.299    &0.268  \\
\hline

$\hat{\rho}_{h=1/128}$    &0.736     &0.735     &0.342   &0.309   &0.311    &0.276  \\
\hline
\end{tabular}\label{IBSR-PrSD-2-sweeps}
\end{table}

Thus, we again  consider solving the Schur complement system by applying a multigrid $W(1,1)$-cycle. Table \ref{IBSR-inner-W-Dohrmann} shows that this IBSR is seen to be  effective, requiring 1 to 4 $W(1,1)$ cycles on the Schur complement system to match the convergence seen in Table \ref{exact-BSR-W-Dohrmann}. Again, $W(1,1)$ cycles seem to be the most cost effective option for the approximate Schur complement.
\begin{table}[H]
 \caption{$W$-cycle convergence factors, $\hat{\rho}_h$, for IBSR  with inner $W(1,1)$-cycle for PrSD and $(\alpha,\omega,\omega_J)=(6/5,16/15,1.1)$. In brackets, minimum value of  the number of inner $W(1,1)$-cycles that achieves the same convergence factors as those of LFA predictions, $\rho_h$, for exact BSR.}
\centering
\begin{tabular}{|l|l|l|l|l|l|l|}
\hline
\hline
Cycle  & $W(0,1)$  & $W(1,0)$   &$W(1,1)$   & $W(1,2)$  &$W(2,1)$   &$W(2,2)$  \\
\hline
 $\rho_{h=1/128}$                   &0.673     & 0.673     &0.111     &0.079    &0.079     &0.062 \\
\hline
$\hat{\rho}_{h=1/64}$     &0.680(4)     &0.677(1)     &0.112(3)     &0.075(2)   &0.075(2)     &0.059(1)           \\
\hline
$\hat{\rho}_{h=1/128}$     &0.659(1)     &0.662(1)     &0.112(3)     &0.075(2)   &0.075(2)     &0.067(1)          \\
\hline
\end{tabular}\label{IBSR-inner-W-Dohrmann}
\end{table}

\subsection{Stabilized discretizations with Uzawa relaxation}
\label{sec:Uzawa}
Multigrid methods using Uzawa relaxation schemes  \cite{MR2833487,MR3217219,MR3580776}  are popular approaches  due to their low cost per iteration. We consider Uzawa relaxation as a simplification of BSR, determining the update as the (weighted) solution of
\begin{equation*}
   M\delta x=  \begin{pmatrix}
      \alpha D & 0\\
     B & -\hat{S}\\
    \end{pmatrix}
    \begin{pmatrix} \delta \mathcal{U} \\ \delta p\end{pmatrix}
  =\begin{pmatrix} r_{\mathcal{U}} \\ r_{p}\end{pmatrix},
\end{equation*}
where $\alpha D$ is an approximation to $A$ and $-\hat{S}$ is an approximation of the Schur complement, $-BA^{-1}B^{T}-C$.

Here, we consider an analogue to exact BSR with $ D={\rm diag}(A)$.  The choice of $\hat{S}$ is discussed later. In this setting, we observe that  minimizing the LFA smoothing
factor does not minimize the LFA convergence factor. Thus,  we  consider minimizing
the two-grid convergence factor numerically for  $\nu_1+\nu_2=1$ and $\nu_1=\nu_2=1$ with rediscretization coarsening, and compare with measured multigrid performance.

We consider three approximations to the Schur complement, starting from the true approximate Schur complement, $C+B(\alpha {\rm diag}(A))^{-1}B^{T}$. Motivated by the stable finite-element case, we also consider replacing $B(\alpha {\rm diag}(A))^{-1}B^{T}$ in this matrix by a weighted mass matrix, yielding $\hat{S}=C+\delta Q$. Finally, motivated by the finite-difference case and efficiency of implementation, we consider taking $\hat{S}=\sigma h^2I$, for a scalar weight, $\sigma$, to be optimized by the LFA. Note that, due to the constant-coefficient stencils assumed by LFA, this corresponds to using a single sweep of Jacobi to approximate solution of either of the two above approximations.

For the case of $C+B(\alpha {\rm diag}(A))^{-1}B^{T}$, the optimized LFA two-grid convergence factors for $\nu_1+\nu_2=1$ with rediscretization coarsening are 0.428 for PoSD and 0.436  for PrSD. These are notably worse than the BSR smoothing factor of $\frac{1}{3}$, which is achieved for $W(1,0)$ or $W(0,1)$ cycles. Here, $W(1,0)$ cycles reflect this convergence, achieving measured convergence factors of 0.417 for PoSD and 0.526 for PrSD. Increasing the number of relaxation sweeps per iteration yields some improvement in the predicted LFA convergence  factors when optimizing parameters again, but not enough to outperform repeated $W(1,0)$ cycles.

For the mass-matrix-based approximation, $\hat{S}=C+\delta Q$, the optimized two-grid convergence factors for $\nu_1+\nu_2=1$ with rediscretization coarsening are 0.5 for PoSD and 0.417 for PrSD. While poorer convergence might be expected in both cases, the addition of an extra parameter, $\delta$, allows a (slight) improvement for PrSD. In both cases, we observe consistent performance with numerical experiments, achieving convergence factors of 0.493 for PoSD and 0.392 for PrSD using $W(0,1)$ or $W(1,0)$ cycles.

Finally, for the diagonal approximation $\hat{S}=\sigma h^2I$, we achieve notably better performance optimizing with $\nu_1=\nu_2=1$ than for $\nu_1+\nu_2=1$. For PoSD, the optimized two-grid LFA convergence factor is 0.382, while it is 0.497 for PrSD. In practice, we achieve slightly worse convergence factors using $W(1,1)$ cycles with rediscretization coarsening, of 0.531 for PoSD and 0.543 for PrSD. These are both significantly worse than the convergence factors of $\frac{1}{9}$ observed using inexact BSR; however, it must be noted that $W$-cycles on the Schur complement system were needed in that case. A better approximation to inverting the true approximate Schur complement would be to apply multigrid to it, just as was done for IBSR above.  Here, we observe that significant work may be needed to achieve convergence similar to that of Uzawa where the Schur complement is exactly inverted, requiring 10 $W(1,1)$-cycles on the approximate Schur complement to achieve a convergence factor of 0.416 for PoSD and 0.522 for PrSD, suggesting that the Jacobi version of Uzawa is ultimately more efficient.

\subsection{Comparing cost and performance}
For convenience, we denote standard DWJ as DWJ(1) and DWJ with 2 sweeps of Jacobi relaxation on the pressure equation as DWJ(2) in the following.

The above results give a clear comparison of the effectiveness of the multigrid cycles with the considered relaxation schemes, but not of their relative efficiencies.  To translate from effectiveness to efficiency, we must properly account for the cost per iteration of each relaxation scheme.  All schemes assume the residual is already calculated; for the 9-point stencils in $A$, $B$, and stabilization terms, $C$, the cost of a single residual evaluation on a mesh with $n$ points is (roughly) that of $63n$ multiply-add operations, coming from the 7 nonzero blocks in the matrix.  For DWJ(1), the rest of the cost of relaxation is fairly easy to account, requiring one diagonal scaling operation on each of the three components of the solution vector, plus matrix-vector products with the pressure Laplacian, $A_p$, and with both $B$ and $B^T$.  Counting multiply-add operations for these on a grid with $n$ points, we have $3n$ for the diagonal scalings, and $9n$ each for the multiplication with $A_p$ and with $B_x$ and $B_y$ and their transposes, totalling $48n$ multiply-add operations. For DWJ(2), we need 48n multiply-add operators plus the cost of the second sweep. For the second sweep, we need to compute a residual related to $G$, the $(3,3)$ block  of (\ref{DWJ-system-FEM}), and a diagonal scaling.  Note that the cost of the residual is $54n$ ($9n$ each for the multiplication with $A_p$, $C$, and with $B_x$ and $B_y$ and their transposes). In total, the cost of DWJ(2) is $103n$ multiply-add operators.   For IBSR, following (\ref{schur-solution-of-precondtion}), we require two diagonal scaling operations on each of the velocity components, one matrix-vector product with each of $B$ and $B^T$, and 2 or 3 W(1,1) cycles on the pressure variable.  To account for the costs of the W(1,1) cycles, we use the standard cost estimate for W-cycles, as requiring 4 ``Work Units'' per iteration, where a Work Unit is the cost of forming a residual for the pressure equation.  Here, given the 25-point stencil structure seen in Equation (\ref{5-stencil-Jacobi}), each Work Unit requires $25$ multiply-add operations, so the total cost of IBSR with 2 W(1,1) cycles on the Schur complement is $4n + 36n + 200n = 240n$ multiply-add operations (and $340n$ multiply-add operations if 3 W(1,1) cycles are needed).  Finally, Uzawa relaxation with diagonal scaling on the pressure has a cost less than that of DWJ(1), as it requires diagonal scaling again for all three components of the solution, but only one matrix-vector multiplication, with $B$.  These total $21n$ multiply-add operations.

Accumulating the costs of a residual evaluation with these, we have total costs of $111n$ multiply-add operations per sweep of DWJ(1), $166n$ multiply-add operations per sweep of DWJ(2),  $303n$ multiply-add operations per sweep of IBSR with 2 W(1,1) cycles per Schur-complement solve, and $84n$ multiply-add operations per sweep of Uzawa with diagonal scaling.  Considering these relative to one-another, we see that DWJ(1) has a cost of about 4/3 per cycle as Uzawa, that DWJ(2) has a per-cycle cost of about 2 times that of Uzawa, and 1.5 times that of DWJ(1), that IBSR has a per-cycle cost of about 3.6 times that of Uzawa, 2.7 and 1.8 times that of DWJ(1) and DWJ(2), respectively.  The per-cycle convergence factors observed above are 0.35 per cycle for W(1,1) cycles of DWJ(1) for PoSD and 0.44 per cycle for W(1,1) cycles of DWJ(1) for PrSD, 0.11 per cycle for W(1,1) cycles of DWJ(2) and IBSR for both stabilizations, and 0.53 or 0.54 per cycle for W(1,1) cycles with Uzawa.  Comparing efficiencies can now be easily done by appropriately weighting these convergence factors relative to their work: if one iteration costs $W$ times that of another, and yields a convergence factor of $\rho_1$, then we can easily compare $\rho_1^{1/W}$ directly to the second convergence factor, $\rho_2$, to see if the {\it effective} error reduction achieved by the first algorithm in an equal amount of work to the second is better or worse than that achieved by the second.  Comparing DWJ(1) to Uzawa, then, for PoSD, we compare $0.35^{3/4} \approx 0.46$ to $0.53$ and see that DWJ(1) is more efficient.  For PrSD, we compare $0.44^{3/4} \approx 0.54$ and see that DWJ(1) and Uzawa are similarly efficient for PrSD. Comparing DWJ(2) to Uzawa, we compare $0.11^{1/2}\approx 0.33$ to 0.53(0.54), showing that DWJ(2) is much efficient than Uzawa. Comparing DWJ(2) to DWJ(1), we compare $0.11^{1/1.5}\approx 0.23$ to 0.35 (0.45) and see DWJ(2) is more efficient than DWJ(1). Comparing IBSR to Uzawa, we compare $0.11^{1/3.6} \approx 0.54$, and see that it as also comparable in efficiency to the others for the case of PrSD, but slightly less efficient than DWJ(1) for PoSD. DWJ(2) and IBSR have the same per-cycle convergence factor, but the cost of DWJ(2) ($166n$) is less  than IBSR ($303n$). Thus DWJ(2) is more efficient than IBSR.  Overall, DWJ(2) outperforms Uzawa, IBSR and DWJ(1). We note that these results are a little different than those seen for the MAC discretization in \cite{NLA2147}, where IBSR outperforms other schemes. Differences seen in practice (and the influence of factors ignored in the LFA, such as boundary conditions) are important to consider.

An important practical consideration commonly observed in the LFA literature (see, for example, \cite{NLA2147,MR1049395}) is the influence of boundary conditions.  In numerical experiments not shown here, we often see significant degradation in convergence between the results for periodic boundary conditions and those for Dirichlet boundary conditions, particularly for cases with larger numbers of relaxation sweeps per cycle.  For DWJ(2) with $\nu_1 + \nu_2 = 1$, changing from periodic to Dirichlet boundary conditions results in convergence factors increasing from 0.324 reported in Tables \ref{stabilized-R-W-2Jacobi} and \ref{PrSD-W-2Jacobi} to about 0.46 (PrSD) or 0.56 (PoSD) for two-grid cycles, and to about 0.64 (PrSD) and 0.7 (PoSD) for W-cycles.  For IBSR with $\nu_1 + \nu_2 = 1$, however, the degradation is much less, with W-cycle convergence rates of 0.38 for PoSD (still with 2 inner W(1,1)-cycles for the Schur complement system) and 0.35 for PrSD (with $\alpha = 4/5$, $\omega = 32/45$, and 4 W(1,1) cycles with $\omega_J=1.1$ for the Schur complement system).  Clearly this difference in performance is enough to change the balance above, with the added cost of IBSR with inner W-cycles paying off over DWJ(2).

\section{Relaxation for $Q_2-Q_1$ discretization}\label{sec:Q2Q1performance}
As explored in \cite{HM2018LFALaplace}, classical LFA smoothing factor analysis is unreliable for $Q_2$ discretizations, making it unsuitable for analysis of the standard stable $Q_2-Q_1$ discretization of the Stokes equations. Thus, we consider only numerical (``brute force'') optimization of two-grid LFA convergence factors in this setting.

For DWJ, we find optimal convergence factors of 0.619 for $\nu_1+\nu_2=1$ and 0.558 for $\nu_1=\nu_2=1$. While the former is quite comparable to convergence predicted and achieved for both stabilized discretizations with $\nu_1+\nu_2=1$, we see a significant lack of improvement with increased relaxation, in contrast to the equal-order case. The same is observed for multigrid $W$-cycle performance, with $W(1,0)$ convergence measured at 0.620 and $W(1,1)$ convergence measured at 0.510.

For exact BSR, we find optimal convergence factors of 0.551 for $\nu_1+\nu_2=1$ and 0.250 for $\nu_1=\nu_2=1$. While these are slightly larger than the comparable factors of $\frac{1}{3}$ and $\frac{1}{9}$, respectively, for the stabilized discretizations, they still reflect good performance of the underlying method.

At left of Figure \ref{BSR2-contour}, we show the spectral radius of the error-propagation symbol for exact BSR as a function of Fourier frequency, $\boldsymbol{\theta}$, noting that predicted reduction over the high frequencies is not as good as would be needed to equal two-grid convergence in the equal-order case. In order to see how the convergence factor  changes with the parameters $\alpha$ and $\omega$, we display the convergence factor as a function of $\alpha$ and $\omega$ at the right of Figure \ref{BSR2-contour}. The optimal choice, of $\alpha=1.1$ and $\omega=1.05$, occurs in a narrow band of $\omega$ values, but larger range of $\alpha$ values lead to reasonable results.

\begin{figure}[H]
\centering
\includegraphics[width=6.5cm,height=5.5cm]{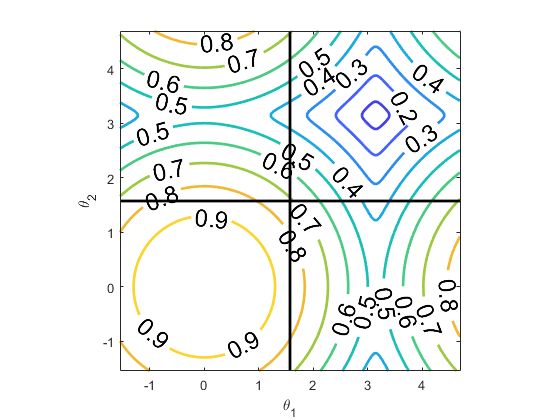}
\includegraphics[width=6.5cm,height=5.5cm]{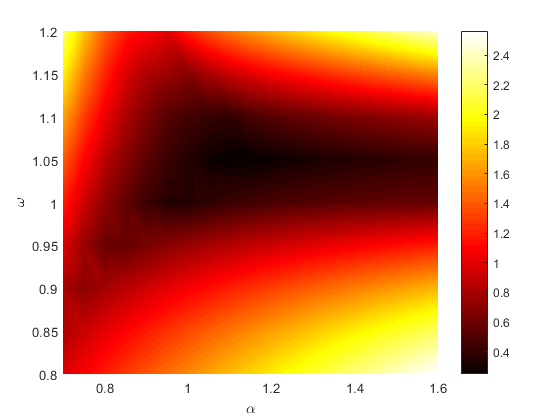}
\caption{At left, the spectral radius of the error-propagation symbol for exact BSR applied to the $Q_2-Q_1$ discretization, as a function of the Fourier mode, $\boldsymbol{\theta}$. At right, the LFA-predicted two-grid convergence factor for BSR applied to the $Q_2-Q_1$ discretization as a function of $\alpha$ and $\omega$, with $(\nu_1,\nu_2)=(1,1)$.} \label{BSR2-contour}
\end{figure}

As always, an inexact solve of the Schur complement system is needed to yield a practical variant of BSR. While 2 sweeps of Jacobi appears sufficient to achieve scalable $W$-cycle convergence when $\nu_1+\nu_2>2$ (see Table \ref{PBC-W-IBSR-2}), we find 3 sweeps are needed to achieve $W(1,1)$ convergence factors of 0.240 (see Table \ref{PBC-W-BSR-3}), in contrast to results in \cite{MR1810326} and for the equal-order discretizations considered here, where a much stronger iteration was needed. Similar results were seen for $V(1,1)$ cycles  when 3 sweeps of Jacobi were used for the Schur complement system.

\begin{table}[H]
 \caption{$W$-cycles convergence factors, $\hat{\rho}_h$, for IBSR with 2 sweeps Jacobi for $Q_2-Q_1$ approximation with rediscretization, compared with LFA two-grid predictions, $\rho_h$, for exact BSR with algorithmic parameters, $\alpha=1.1,\omega=1.05$ and $\omega_J=1.0$.}
\centering
\begin{tabular}{|l|c|c|c|c|c|c|}
\hline
Cycle   & $W(0,1)$  & $W(1,0)$   &$W(1,1)$   & $W(1,2)$  &$W(2,1)$   &$W(2,2)$  \\
\hline
\hline
$\rho_{h=1/128}$                   &4.893     &4.893      &0.249     &0.109      &0.109     &0.090\\
\hline
$\hat{\rho}_{h=1/64}$     & NAN      &NAN     &0.434     & 0.131      &0.130     &0.085       \\
\hline
$\hat{\rho}_{h=1/128}$     &NAN      &NAN     & 0.437    & 0.130    &0.130     &0.085 \\
\hline
\end{tabular}\label{PBC-W-IBSR-2}
\end{table}

\begin{table}[H]
 \caption{$W$-cycles convergence factors, $\hat{\rho}_h$, for IBSR with 3 sweeps Jacobi for $Q_2-Q_1$ approximation with rediscretization, compared with LFA two-grid predictions, $\rho_h$, for exact BSR with algorithmic parameters, $\alpha=1.1,\omega=1.05$ and $\omega_J=1.0$.}
\centering
\begin{tabular}{|l|c|c|c|c|c|c|}
\hline
Cycle   & $W(0,1)$  & $W(1,0)$   &$W(1,1)$   & $W(1,2)$  &$W(2,1)$   &$W(2,2)$  \\
\hline
\hline
$\rho_{h=1/128}$                    &4.893     &4.893      &0.249     &0.109      &0.109     &0.090\\
\hline
$\hat{\rho}_{h=1/64}$     & 491.373     &492.094     &0.240     & 0.104     &0.104     &0.085       \\
\hline
$\hat{\rho}_{h=1/128}$     &NAN      &NAN     & 0.240    & 0.104    &0.104       &0.085 \\
\hline
\end{tabular}\label{PBC-W-BSR-3}
\end{table}

Finally, we consider the same three variants of Uzawa relaxation as examined above for the equal-order case. For $\hat{S}=B(\alpha{\rm diag}(A))^{-1}B^{T}$, the best convergence factor found for $\nu_1+\nu_2=1$ was 0.729, while better convergence was predicted for $\hat{S}=\delta Q$, with factor 0.554. This is to be expected, perhaps, since the $Q_1$ mass matrix is well-known to be a better approximation of the true Schur complement than the classical BSR approximate Schur complement. However, approximating either by a single sweep of Jacobi, yielding $\hat{S}=\sigma h^2I$, gives a convergence factor 0.717. While 2-grid cycles with $\nu_1+\nu_2=1$ match the predicted convergence factor, $W$-cycles did not converge for these parameters.

Comparing, then, the efficiency of inexact BSR and DWJ for the $Q_2-Q_1$ discretization, we see that inexact BSR, where $W(1,1)$ cycles achieve a convergence factor of 0.24 provides roughly the  same reduction as 3 cycles with 1 DWJ sweep per cycle, where LFA predicts $\rho=0.619$. Noting that inexact BSR is relatively more expensive in this case, with cost dominated by the two diagonal scalings per sweep on the $Q_2$ velocity degrees of freedom, we suggest a proper implementation study is needed to determine which, if either, provides best performance in practice.

\section{Conclusion}
\label{sec:concl-FEM}
In this paper, LFA is presented for block-structured relaxation schemes for stabilized and stable finite-element discretizations of
the Stokes equations. The convergence and smoothing factors  exhibited  here provide  optimized parameters for DWJ with one or two sweeps of Jacobi relaxation on the pressure equation and  BSR for
the stabilized discretizations. The convergence of (inexact) BSR clearly outperforms multigrid with both standard DWJ and Uzawa relaxation. However, standard DWJ can be improved by additional relaxation on the pressure equation, and the improved version is more efficient than IBSR.  While the LFA smoothing
factor loses its predictivity of  the two-grid convergence factor for the stable $Q_2-Q_1$ discretization and for Uzawa relaxation for both
stabilized and   stable discretizations, the two-grid LFA convergence factor can still provide useful predictions. We consider as well the inexact case for BSR, with Jacobi iterations or multigrid cycles used to approximate solution of the Schur complement system, as is suitable for use
on modern  parallel and graphics processing unit (GPU) architectures. From numerical experiments, we see that inexact BSR can be as good as the exact iteration for solving the Stokes equations.  The analysis and LFA predictions demonstrated here offer good insight into the use of block-structured
relaxation for other types of saddle-point problems, which will be considered in future work.

\section*{Acknowledgements}
The work of S.M. was partially supported by an NSERC discovery grant.

\bibliographystyle{elsarticle-num}
\bibliography{FEM_Stokes}

\end{document}